\documentclass[reqno]{amsart}
\usepackage{amsmath,amsfonts,amsthm,amssymb}
\usepackage{enumitem}

\usepackage{todonotes}

\usepackage[normalem]{ulem}

\usepackage{bbm}
\usepackage{hyperref}

\allowdisplaybreaks   

\theoremstyle{plain}
\newtheorem{de}{Definition}[section]
\newtheorem{lem}[de]{Lemma}
\newtheorem{prop}[de]{Proposition}
\newtheorem{cor}[de]{Corollary}
\newtheorem{thm}[de]{Theorem}

\theoremstyle{definition}
\newtheorem{rem}[de]{Remark}

\numberwithin{equation}{section}

\newcommand{\eul}{e}
\newcommand{\imu}{\mathrm{i}}
\newcommand{\dd}{\,\mathrm{d}}
\renewcommand{\epsilon}{\varepsilon}

\renewcommand{\Im}{\operatorname{Im}}
\renewcommand{\Re}{\operatorname{Re}}
\newcommand{\supp}{\operatorname{supp}}

\newcommand{\R}{{\mathbb{R}}}
\newcommand{\N}{{\mathbb{N}}}

\newcommand{\Z}{{\mathbb{Z}}}

\newcommand{\PP}{{\mathbb{P}}}

\newcommand{\cL}{{\mathcal{L}}}
\newcommand{\cF}{{\mathcal{F}}}
\newcommand{\cA}{{\mathcal{A}}}

\newcommand{\cI}{{\mathcal{I}}}

\newcommand{\om}{\omega}

\newcommand{\ph}{\varphi}

\begin{document}
 \title[On the almost sure scattering for the 4D cubic NLW]{On the almost sure scattering for the energy-critical cubic wave equation with supercritical data}
 
 \author[M.~Spitz]{Martin Spitz}
 \address[M.~Spitz]{Fakult\"at f\"ur Mathematik, Universit\"at Bielefeld, Postfach 10 01 31, 33501
  Bielefeld, Germany}
\email{mspitz@math.uni-bielefeld.de}
\keywords{Nonlinear wave equation, random initial data, almost sure scattering}

\begin{abstract}
	In this article we study the defocusing energy-critical nonlinear wave equation on $\R^4$ with scaling supercritical data. We prove almost sure scattering for randomized initial data in $H^s(\R^4) \times H^{s-1}(\R^4)$ with $\frac{5}{6} < s < 1$. The proof relies on new probabilistic estimates for the linear flow of the wave equation with randomized data, where the randomization is based on a unit-scale decomposition in frequency space, a decomposition in the angular variable, and a unit-scale decomposition of physical space. In particular, we show that the solution to the linear wave equation with randomized data almost surely belongs to $L^1_t L^\infty_x$.
\end{abstract}

\maketitle

\section{Introduction and main result}
\label{sec:Introduction}

We consider the cubic defocusing nonlinear wave equation in four space dimensions
\begin{equation}
	\label{eq:CubicNLW}
	\begin{aligned}
		-\partial_t^2 u + \Delta u &= u^3 \hspace{4em} \text{on } \R \times \R^4, \\
			(u(0), \partial_t u(0)) &= (f_0, f_1)
	\end{aligned}
\end{equation}
with initial data $(f_0, f_1) \in H^s(\R^4) \times H^{s-1}(\R^4)$. The energy
\begin{equation}
	\label{eq:DefEnergy}
	E(u(t)) = \int_{\R^4} \frac{1}{2} |\nabla u(t)|^2 + \frac{1}{2} |\partial_t u(t)|^2 + \frac{1}{4} u^4(t) \dd x
\end{equation}
is conserved for sufficiently regular solutions of~\eqref{eq:CubicNLW}. Since both the energy and equation~\eqref{eq:CubicNLW} are invariant under the scaling
\begin{equation*}
	u(t,x) \mapsto \lambda u(\lambda t, \lambda x),
\end{equation*}
problem~\eqref{eq:CubicNLW} is called energy-critical. Moreover, the corresponding scaling critical regularity is $s_c = 1$ and we refer to $\dot{H}^1(\R^4) \times L^2(\R^4)$ as the energy space. The term defocusing refers to the plus sign in front of the nonlinearity in~\eqref{eq:CubicNLW}.

The purpose of this article is to study the long-time behavior of solutions of~\eqref{eq:CubicNLW} below the scaling-critical regularity. We prove that for initial data $(f_0, f_1) \in H^s(\R^4) \times H^{s-1}(\R^4)$ with $\frac{5}{6} < s < 1$ which are randomized with respect to a decomposition in frequency space, in the angular variable, and in physical space, the solutions of~\eqref{eq:CubicNLW} almost surely scatter.

The energy-critical nonlinear wave equation (NLW) has been subject of extensive research. Local wellposedness of~\eqref{eq:CubicNLW} was shown in ~\cite{LS1995} for subcritical and critical regularities $s \geq 1$. The global theory, in particular global wellposedness, scattering, and control of scattering norms by the energy, was developed in a series of works by several authors, see~\cite{Str1988, Gr1990, GSV1992, SS1993, SS1994,BG1998, BG1999, Na1999, Tao2006} and the references therein. They have particularly shown that for any $(u_0, u_1) \in \dot{H}^1(\R^4) \times L^2(\R^4)$, equation~\eqref{eq:CubicNLW} with initial data $(u_0, u_1)$ has a unique global solution
\begin{align*}
	(u, \partial_t u) \in (C(\R, \dot{H}^1(\R^4)) \cap L^3(\R, L^6(\R^4))) \times C(\R, L^2(\R^4)).
\end{align*}
Moreover, there is a non-decreasing function $K \colon [0,\infty) \rightarrow [0,\infty)$ such that
\begin{align*}
	\|u\|_{L^3_t L^6_x} \leq K(E(u_0, u_1)), 
\end{align*}
where
\begin{align*}
	E(u_0, u_1) = \int_{\R^4} \frac{1}{2} |\nabla u_0|^2 + \frac{1}{2} u_1^2 + \frac{1}{4} u_0^4 \dd x.
\end{align*}
Consequently, the unique global solution of~\eqref{eq:CubicNLW} scatters both forward and backward in time.

Below the scaling-critical regularity, i.e. for $s < 1$, problem~\eqref{eq:CubicNLW} is known to be ill-posed, see~\cite{CCT2003}. However, ill-behaved solutions of~\eqref{eq:CubicNLW} in the scaling-supercritical regime might rather be exceptional, which leads to the question if there still are large sets of generic initial data for which unique local, global or scattering solutions exist. We study this question using randomization.


\subsection{Randomization}
\label{subsec:Randomization}

Random dispersive partial differential equations have attracted a lot of interest in recent years emanating from the seminal works~\cite{B1994, B1996} and~\cite{BT2008I, BT2008II}. A large body of literature has developed by now and we only review the results which are most relevant for this work.

Almost sure global existence for the defocusing energy-critical nonlinear wave equation with initial data $(f_0, f_1) \in H^s(\R^d) \times H^{s-1}(\R^d)$ was proven in~\cite{P2017} for $s > 0$ in dimension $d = 4$, for $s \geq 0$ in dimension $d = 5$, and in~\cite{OP2016} for $s > \frac{1}{2}$ in the case $d = 3$. Concerning the long-time behavior, the first almost sure scattering result for a scaling-supercritical dispersive equation was established in~\cite{DLM2020} for the energy-critical nonlinear wave equation in four space dimensions for $\frac{1}{2} < s < 1$ and initial data $(f_0, f_1)$ which is radially symmetric before the randomization. This result was then improved to $0 < s < 1$ in~\cite{DLM2019}, still for randomized radially symmetric initial data. In dimension $d = 3$ almost sure scattering for the energy-critical NLW with randomized radially symmetric initial data was shown in~\cite{Br2020}. We note that in a parallel line of research, similar results were also obtained for the energy-critical respectively cubic nonlinear Schr{\"o}dinger equation, see~\cite{BOP2015I, BOP2015II, B2019, OOP2019, KMV2019, DLM2019, SSchr2021, C2021, SSW2021, SSW2021II}.

The first almost sure scattering result without the radial symmetry assumption was proven in~\cite{Br2021} for the energy-critical NLW in dimension $d = 4$, i.e.~\eqref{eq:CubicNLW}, for $\frac{11}{12} < s < 1$. We note that in~\cite{Br2021} a different randomization than in~\cite{DLM2020, DLM2019} for radially symmetric data was used. In fact, most of the aforementioned results (except~\cite{Br2020, SSchr2021, SSW2021II}) relied on the so-called Wiener randomization, which is based on a unit-scale decomposition of frequency space (see Subsection~\ref{subsec:UnitScaleFrequencySpace} below for details).

However, there are several randomization procedures, each having its own advantages. In~\cite{BT2008I, BT2008II}, where a nonlinear wave equation was considered on compact manifolds, the randomization was based on a decomposition in an orthonormal basis of eigenfunctions of the Laplacian. Similar approaches were then also used on $\R^d$ for problems which could be transformed into a setting where an orthonormal basis of eigenfunctions of the Laplacian exists, see e.g.~\cite{dS2013, D2012, PRT2014} and the references therein.

The Wiener randomization was then introduced in~\cite{BOP2015I, BOP2015II} and~\cite{LM2014} on the Euclidean space. It provides a unit-scale Bernstein inequality which improves the range of Strichartz estimates from the deterministic setting. Rather recently, a randomization with respect to a unit-scale decomposition of physical space was proposed in~\cite{M2019} (see Subsection~\ref{subsec:PhysicalSpaceRand} below) in the context of the final-state problem of the mass-subcritical NLS in $L^2$ , see also~\cite{NY2019}. Roughly speaking, this physical-space randomization gives access to the dispersive estimate although the data only belongs to $L^2$. Another randomization with respect to a certain decomposition in the angular variable (see Subsection~\ref{subsec:AngularDecompRand} below for details) was introduced in~\cite{BK2021}. With this randomization (almost) the same range of Strichartz estimates as for radial data is available.

In~\cite{Br2021} the Wiener randomization was combined with the physical-space randomization to a microlocal randomization based on a unit-scale decomposition both in frequency and physical space. Combinations of different randomization procedures were also employed in~\cite{BK2021}, where the angular decomposition was merged with the Wiener randomization and a randomization in the radial variable to prove almost sure global wellposedness for a wave-maps type nonlinear wave equation. Recently, the author used a randomization with respect to the unit-scale decomposition of frequency space, the decomposition in the angular variable, and the unit-scale decomposition of physical space to prove almost sure scattering for the energy-critical cubic Schr{\"o}dinger equation without the radial symmetry assumption for the initial data in~\cite{SSchr2021}.

In this work we show that using the same randomization, i.e. adding the decomposition in the angular variable to the microlocal randomization from~\cite{Br2021}, we can lower the regularity threshold for almost sure scattering of~\eqref{eq:CubicNLW} from $\frac{11}{12} < s < 1$ in~\cite{Br2021} to $\frac{5}{6} < s < 1$. The proof heavily relies on the improved probabilistic estimates for the linear flow of the wave equation as we provide an $L^1_t L^\infty_x$-estimate for linear solutions with randomized data. We discuss the effects of this randomization in more detail after the statement of the main result in Subsection~\ref{subsec:MainResult} below. However, we first provide the precise definition of the randomization. In order to facilitate further applications of this randomization, we present it in general dimension $d \geq 2$ and only specialize to $d = 4$ when we prove the almost sure scattering result for~\eqref{eq:CubicNLW}.

\subsection{Unit scale decomposition in frequency space}
\label{subsec:UnitScaleFrequencySpace}

Let $\phi \in C_c^\infty(\R^d)$ be non-negative and even with $\phi(x) = 1$ if $x \in [-\frac{2}{3},\frac{2}{3}]^d$ and $\phi(x) = 0$ if $x \notin [-\frac{3}{4}, \frac{3}{4}]^d$. Setting
\begin{equation}
\label{eq:DefPsij}
	\psi_j(\xi) = \frac{\phi(\xi - j)}{\sum_{m \in \Z^d} \phi(\xi - m)}
\end{equation}
for all $\xi \in \R^d$ and $j \in \Z^d$, we obtain a smooth partition of unity $\{\psi_j \colon j \in \Z^d\}$. We define the operators $P_j$ by
\begin{equation}
	\label{eq:DefPj}
	P_j f = \cF^{-1}(\psi_j \hat{f})
\end{equation}
for all $j \in \Z^d$, where $\cF f = \hat{f}$ denotes the Fourier transform of $f$. We thus get the unit scale decomposition in frequency space
\begin{equation}
	\label{eq:DecompositionPj}
	f = \sum_{j \in \Z^d} P_j f
\end{equation}
for all $f \in H^s(\R^d)$. Randomizing with respect to this decomposition leads to the Wiener randomization, see e.g.~\cite{BOP2015I, BOP2015II, LM2014}. We further note that
\begin{equation}
	\label{eq:SymPsij}
	\psi_j(-\xi) = \psi_{-j}(\xi)
\end{equation}
for all $\xi \in \R^d$ and $j \in \Z^d$ since $\phi$ is even.

\subsection{Unit scale decomposition in physical space}
\label{subsec:PhysicalSpaceRand}

We use the same partition of unity as above in order to decompose a function in physical space. However, to ease the distinction between the decomposition in physical and in frequency space, we employ a different notation. We set
\begin{equation}
	\label{eq:DefPhysicalSpacePartition}
	\ph_i = \psi_i
\end{equation}
for all $i \in \Z^d$, leading to the unit scale decomposition in physical space
\begin{align*}
	f = \sum_{i \in \Z^d} \ph_i f
\end{align*}
for all $f \in H^s(\R^d)$. Randomizing with respect to this decomposition yields the physical-space randomization, see~\cite{M2019,NY2019,S2020}.

\subsection{Decomposition in the angular variable with respect to a good frame}
\label{subsec:AngularDecompRand}
Finally, we introduce the decomposition in the angular variable, closely following~\cite{BK2021,S2020}. We begin by recalling that the eigenfunctions of the Laplacian on the sphere are the spherical harmonics, i.e. the restriction to $S^{d-1}$ of the homogeneous harmonic polynomials. The space $E_k$ of spherical harmonics of degree $k$ has dimension
\begin{align*}
	N_k = \binom{d + k - 1}{k} - \binom{d + k - 3}{k - 2}
\end{align*}
for every $k \in \N_0$. We fix an orthonormal basis
\begin{align*}
		\{b_{k,l} \in L^2(S^{d-1}) \colon l \in \{1, \ldots, N_k\}, k \in \N_0\}
	\end{align*}
	of $L^2(S^{d-1})$, consisting of eigenfunctions of $\Delta_{S^{d-1}}$, such that there exists a constant $C > 0$ with
\begin{equation}
		\label{eq:BoundednessGoodFrame}
		\|b_{k,l}\|_{L^q(S^{d-1})} \leq \begin{cases} 
										C \sqrt{q} \qquad &\text{if } q < \infty, \\
										C \sqrt{\log(k)} &\text{if } q = \infty
									\end{cases}
	\end{equation}		
	for all $l \in \{1, \ldots, N_k\}$, $k \in \N$, and $q \in [2,\infty]$. The existence of such a frame follows from Th{\'e}or{\`e}me~6 and Proposition~3.2 in~\cite{BL2013}, see also~\cite[Theorem~1.1]{BK2021} and~\cite{BL2014}. We call an orthonormal basis $\{b_{k,l} \colon l \in \{1, \ldots, N_k\}, k \in \N_0\}$ as above a \emph{good frame}. Given a function $f \in H^s(\R^d)$, we can now develop every Littlewood-Paley block $P_M f$ in the good frame, which yields a decomposition of $f$. The randomization with respect to this decomposition was introduced in~\cite{BK2021}, see also~\cite{S2020}.
		
\subsection{Definition of the randomization}	
	Next let $s \in \R$ and take $f \in H^s(\R^d)$. For our randomization we combine the three decompositions from above, where we first apply the unit-scale decomposition in physical space, then the decomposition in the angular variable with respect to the good frame, and finally the unit-scale decomposition in frequency space. To that purpose, we first rescale the Littlewood-Paley blocks of $\ph_i f$ to unit frequency by setting
	\begin{equation}
		\label{eq:DefgM}
		g^M_i = (P_M (\ph_i f))(M^{-1} \cdot)
	\end{equation}
	for all $i \in \Z^d$ and $M \in 2^\Z$, where $P_M$ denotes the standard Littlewood-Paley operator introduced in Section~\ref{sec:NotPrelim} below. After transition to polar coordinates, we expand the Fourier transform of $g^M_i$ in the good frame, yielding
	\begin{equation}
		\label{eq:hatgMiInGoodFrame}
		\hat{g}^M_i(\rho \theta) = \sum_{k = 0}^\infty \sum_{l = 1}^{N_k} \hat{c}^{M,i}_{k,l}(\rho) b_{k,l}(\theta),
	\end{equation}
	where all coefficients $\hat{c}^{M,i}_{k,l}$ are supported in $(\frac{1}{2},2)$. Theorem~3.10 from~\cite{SW71} thus gives the representation
	\begin{equation}
		\label{eq:RepresentationgMi}
		g^M_i(r \theta) = \sum_{k = 0}^\infty \sum_{l = 1}^{N_k} a_k r^{-\frac{d-2}{2}} b_{k,l}(\theta) \int_0^\infty \hat{c}^{M,i}_{k,l}(\rho) J_{\frac{d+2k-2}{2}}(r \rho) \rho^{\frac{d}{2}} \dd \rho,
	\end{equation}
	where $a_k = \imu^k (2\pi)^{-\frac{d}{2}}$ and $J_\mu$ denotes the Bessel function
	\begin{align*}
 		J_\mu(t) = \frac{(\frac{t}{2})^\mu}{\Gamma(\frac{2\mu+1}{2}) \Gamma(\frac{1}{2})}\int_{-1}^1 \eul^{\imu t s} (1-s^2)^{\frac{2\mu-1}{2}} \dd s
 	\end{align*}
 	for all $t > 0$ and $\mu > -\frac{1}{2}$. By Plancherel's theorem and~\eqref{eq:hatgMiInGoodFrame} we also get
	\begin{equation}
		\label{eq:L2SummabilityAngDec}
		\|g^M_i\|_{L^2_x}^2 \sim \sum_{k = 0}^\infty \sum_{l = 1}^{N_k} \| \hat{c}^{M,i}_{k,l}\|_{L^2(\rho^{d-1} \dd \rho)}^2.
	\end{equation}
	Scaling back to frequency $M$, we arrive at
\begin{equation}
	\label{eq:DecompSphHarm}
	P_M(\ph_i f)(r \theta) = \sum_{k = 0}^\infty \sum_{l = 1}^{N_k} a_k M^{-\frac{d-2}{2}} r^{-\frac{d-2}{2}} b_{k,l}(\theta) \int_0^\infty \hat{c}^{M,i}_{k,l}(\rho) J_{\frac{d+2k-2}{2}}(M r \rho) \rho^{\frac{d}{2}} \dd \rho.
\end{equation}
Recalling that $(\ph_i)_{i \in \Z^d}$ is a partition of unity and applying the unit-scale decomposition in frequency space, we thus obtain the decomposition
\begin{equation}
\label{eq:DecPMf}
	P_M f = \sum_{i,j \in \Z^d} \sum_{k = 0}^\infty \sum_{l = 1}^{N_k} a_k M^{-\frac{d-2}{2}} P_j \Big[r^{-\frac{d-2}{2}} b_{k,l}(\theta) \int_0^\infty \hat{c}^{M,i}_{k,l}(\rho) J_{\frac{d+2k-2}{2}}(M r \rho) \rho^{\frac{d}{2}} \dd \rho \Big].
\end{equation}

We now introduce the randomization with respect to the above decomposition. Fix an index set $\cI$ such that $\Z^d = \cI \dot\cup \{0\} \dot\cup (-\cI)$. For each $M \in 2^\Z$ we pick a sequence $(X^M_{i,j,k,l})_{i \in \Z^d, j \in \cI \cup \{0\}, l \in \{1, \ldots, N_k\}, k \in \N_0}$ of complex-valued random variables on a probability space $(\Omega, \cA, \PP)$ such that 
\begin{align*}
	(X^M_{i,0,k,l}, \Re(X^M_{i,j,k,l}), \Im(X^M_{i,j,k,l}))_{i \in \Z^d, j \in \cI, l \in \{1, \ldots, N_k\}, k \in \N_0}
\end{align*}
are independent, real-valued, mean-zero random variables and that there is a constant $c > 0$ satisfying
\begin{align*}
	\int_{\R} e^{\gamma x} \dd \mu^{M,n}_{i,j,k,l}(x) \leq e^{c \gamma^2}
\end{align*}
for all $\gamma \in \R$, $l \in \{1, \ldots, N_k\}$, $k \in \N_0$, $i \in \Z^d$,  $j \in \cI \cup \{0\}$, and $n \in \{1,2\}$, where $\mu^{M,1}_{i,j,k,l}$ denotes the distribution of $\Re(X^M_{i,j,k,l})$ and $\mu^{M,2}_{i,j,k,l}$ the distribution of $\Im(X^M_{i,j,k,l})$. We then set $X_{i,-j,k,l}^M = \overline{X}_{i,j,k,l}^M$ for all $j \in \cI$. We note that mean-zero Gaussian random variables with uniformly bounded variances, standard Bernoulli variables, and mean-zero random variables with compactly supported distributions satisfy this assumption.

In view of the decomposition~\eqref{eq:DecPMf}, we then define the \emph{randomization} of $f$ as
\begin{align}
	\label{eq:DefRandomization}
	f^\om &= \sum_{M \in 2^\Z} \sum_{i,j \in \Z^d} \sum_{k = 0}^\infty \sum_{l = 1}^{N_k} X^M_{i,j,k,l}(\om) a_k M^{-\frac{d-2}{2}} \nonumber\\
	&\hspace{10em} \cdot P_j \Big[r^{-\frac{d-2}{2}} b_{k,l}(\theta) \int_0^\infty \hat{c}^{M,i}_{k,l}(\rho) J_{\frac{d+2k-2}{2}}(M r \rho) \rho^{\frac{d}{2}} \dd \rho \Big],
\end{align}
which is understood as the limit in $L^2(\Omega, H^s(\R^d))$.

\begin{rem}
	\begin{enumerate}
		\item We point out that this randomization does not regularize in the sense of higher Sobolev regularity in general. In fact, assume that $(\Omega, \cA, \PP)$ is the product space of $(\Omega_i, \cA_i, \PP_i)$, $i = 1,2,3$, and that $X^M_{i,j,k,l}(\om) = X_j(\om_1) X_{k,l}^M(\om_2) X_i(\om_3)$ for all $\om = (\om_1, \om_2, \om_3) \in \Omega$, where $(X_j)_j$, $(X^M_{k,l})_{M,k,l}$, and $(X_i)_i$ are independent and identically distributed real-valued Gaussian random variables on $\Omega_1$, $\Omega_2$, and $\Omega_3$, respectively. Employing adaptions of the proof of~\cite[Lemma~B.1]{BT2008I} successively in $\omega_3$, $\omega_2$, and $\omega_1$, one can show that for $f \in H^s(\R^d)$ which does not belong to $H^{s'}(\R^d)$, $f^\om$ is not an element of $H^{s'}(\R^d)$ almost surely for every $s' > s$.
		\item The introduction of the index set $\cI$ has the effect that the randomization $f^\om$ of a real-valued function $f$ is again real-valued. In fact, assume that $f$ is real-valued. Property~\eqref{eq:SymPsij} of the partition $(\psi_j)_{j \in \Z^d}$ implies that
		\begin{align*}
			\overline{P_j f} = P_{-j} f
		\end{align*}
		for all $j \in \Z^d$. If $(c_j)_{j \in \Z^d}$ is a complex-valued sequence with $c_{-j} = \overline{c}_j$ for all $j \in \cI \cup \{0\}$, we thus obtain that
		\begin{align*}
			\sum_{j \in \Z^d} c_j P_j f = c_0 P_0 f + \sum_{j \in \cI} (c_j P_j f + \overline{c_j P_j f}) 
			= c_0 P_0 f + \sum_{j \in \cI} 2 \Re(c_j P_j f)
		\end{align*}
		is real-valued. Since $P_M(\ph_i f)$ is real valued for all $i \in \Z^d$ and $M \in 2^\Z$, each summand on the right-hand side of~\eqref{eq:DecompSphHarm} is real-valued and hence $f^\om$ is real-valued for real-valued $f$.
	\end{enumerate}
\end{rem}

\subsection{Main result}
\label{subsec:MainResult}

We set
\begin{equation}
	\label{eq:DefS}
	S(t)(f_0, f_1) = \cos(t |\nabla|) f_0 + \frac{\sin(t |\nabla|)}{|\nabla|} f_1
\end{equation}
for the solution of the linear wave equation with initial data $(f_0, f_1)$. The main result of this work shows almost sure scattering for solutions of~\eqref{eq:CubicNLW} with randomized initial data from $H^s(\R^4) \times H^{s-1}(\R^4)$ with $\frac{5}{6} < s < 1$.

\begin{thm}
	\label{thm:AlmostSureScattering}
	Let $\frac{5}{6} < s < 1$, $(f_0, f_1) \in H^s(\R^4) \times H^{s-1}(\R^4)$, and $(f_0^\om, f_1^\om)$ be the randomization from~\eqref{eq:DefRandomization}. Then for almost every $\omega \in \Omega$ there exists a unique global solution 
	\begin{align*}
		(u, \partial_t u) \in (S(t)(f_0^\om, f_1^\om), \partial_t S(t)(f_0^\om, f_1^\om)) + C(\R, \dot{H}^1(\R^4) \times L^2(\R^4))
	\end{align*}
	of the energy-critical cubic nonlinear wave equation
	\begin{equation}
	\label{eq:CubicNLWRand}
	\begin{aligned}
		-\partial_t^2 u + \Delta u &= u^3 \hspace{4em} \text{on } \R \times \R^4, \\
			(u(0), \partial_t u(0)) &= (f_0^\om, f_1^\om),
	\end{aligned}
\end{equation}
	 which scatters both forward and backward in time, i.e. there exist $(v_0^{\pm}, v_1^{\pm}) \in \dot{H}^{1}(\R^4) \times L^2(\R^4)$ such that 
	\begin{align*}
		\lim_{t \rightarrow \pm \infty}\| (u(t), \partial_t u(t)) - (u_{\mathrm{lin},\pm}^\om(t),\partial_t u_{\mathrm{lin},\pm}^\om(t)) \|_{\dot{H}^1 \times L^2} = 0,
	\end{align*}
	where
	\begin{align*}
		u_{\mathrm{lin},\pm}^\om(t) = S(t)(f_0^\om + v_0^\pm, f_1^\om + v_1^\pm).
	\end{align*}
\end{thm}
Uniqueness in the above theorem means that 
\begin{equation}
	\label{eq:Defv}
	(u, \partial_t u) - (S(t)(f_0^\om, f_1^\om), \partial_t S(t)(f_0^\om, f_1^\om))
\end{equation}
is unique in $(C(\R, \dot{H}^1(\R^4)) \cap L^3(\R, L^6(\R^4))) \times C(\R,L^2(\R^4))$.

For the proof of Theorem~\ref{thm:AlmostSureScattering} we set
\begin{align*}
	v(t) = u(t) - S(t)(P_{>8} f_0^\om, P_{>8} f_1^\om).
\end{align*}
Note that $u$ is a solution of~\eqref{eq:CubicNLWRand} if and only if $v$ solves the forced cubic nonlinear wave equation
\begin{equation}
	\label{eq:CubicNLWForced}
	\begin{aligned}
		-\partial_t^2 v + \Delta v &= (v+F)^3 \hspace{4em} \text{on } \R \times \R^4, \\
			(v(0), \partial_t v(0)) &= (v_0, v_1)
	\end{aligned}
\end{equation}
with forcing $F(t) = S(t)(P_{>8} f_0^\om, P_{>8} f_1^\om)$ and initial data $(v_0, \!v_1) = (P_{\leq 8} f_0^\om, P_{\leq 8} f_1^\om)$.
We use the overall strategy developed in~\cite{DLM2020} and successfully applied in~\cite{KMV2019, DLM2019, Br2021, SSchr2021} to obtain almost sure scattering results for the energy-critical cubic nonlinear wave and Schr{\"o}dinger equation. The strategy consists in developing a suitable local wellposedness and perturbation theory for the forced equation~\eqref{eq:CubicNLWForced}, which is then combined with the existing deterministic theory for~\eqref{eq:CubicNLW} to derive a scattering result for~\eqref{eq:CubicNLWForced} conditioned on an a priori bound of the energy $E(v)$ of $v$. Note that the energy of $v$ is not conserved as $v$ is not a solution of the energy-critical NLW~\eqref{eq:CubicNLW}. The task of proving scattering for the solution $v$ of~\eqref{eq:CubicNLWForced} thus reduces to bound the energy of $v$.

An a priori bound of the energy of $v$ can be derived by computing the time derivative $\partial_t E(v(t))$ and estimating the resulting terms. It turns out that the most difficult one to control is $\|F v^2 \partial_t v\|_{L^1_t L^1_x}$. If one can only use the energy to bound the factors involving $v$, one is forced to control $F$ in $L^1_t L^\infty_x$. On the other hand, it has been known from~\cite{BT2014} that $F \in L^3_t L^6_x \cap L^1_t L^\infty_x$ implies an a priori bound of the energy of $v$, see~\cite[Remark~1.8]{DLM2020}. While the $L^3_t L^6_x$-bound for the linear evolution of the wave equation with randomized data is easily obtained, the global $L^1_t L^\infty_x$-estimate has been the main obstacle ever since.

With such an $L^1_t L^\infty_x$-bound for the linear solution of the wave equation with Wiener-randomized data not available (cf.~\cite[Remark~1.8]{DLM2020}), the authors in~\cite{DLM2020} not only employed the energy but also a Morawetz-type estimate and a double bootstrap argument in order to control the energy of $v$. However, this approach required to control spatially weighted norms of $F$, leading to the radial symmetry assumption on $(f_0, f_1)$ so that the linear flow with randomized data satisfied the assumptions on $F$ almost surely. This approach yielded almost sure scattering for randomized radially symmetric data $(f_0, f_1) \in H^s(\R^4) \times H^{s-1}(\R^4)$ for $\frac{1}{2} < s < 1$. Using local energy decay, this result was improved to $0 < s < 1$ in~\cite{DLM2019}.

In order to remove the radial symmetry assumption, the Wiener randomization was combined with the physical-space randomization in~\cite{Br2021}. With a global $L^1_t L^\infty_x$-estimate still not available, a careful wave packet analysis was performed to prove almost sure scattering for randomized data $(f_0, f_1) \in H^s(\R^4) \times H^{s-1}(\R^4)$ with $\frac{11}{12} < s < 1$.

In this article, we further add the angular randomization with respect to a good frame to the randomization from~\cite{Br2021}. We show that the resulting randomization~\eqref{eq:DefRandomization} allows us to bound the linear evolution of the wave equation with randomized data almost surely in $L^1_t L^\infty_x$. As noted above, such a bound then implies almost sure scattering of solutions to~\eqref{eq:CubicNLW}. Moreover, we obtain this $L^1_t L^\infty_x$-bound almost surely for linear solutions of the wave equation with randomized initial data from $H^s(\R^4) \times H^{s-1}(\R^4)$ with $\frac{5}{6} < s < 1$, where the original data need not be radially symmetric. Consequently, we improve the regularity threshold for almost sure scattering of~\eqref{eq:CubicNLW} from $s > \frac{11}{12}$ in~\cite{Br2021} to $s > \frac{5}{6}$.

The proof of Theorem~\ref{thm:AlmostSureScattering} is thus based on new probabilistic estimates for the linear flow of the wave equation with randomized initial data, the main novelty being the almost sure $L^1_t L^\infty_x$-estimate of $e^{\pm \imu t |\nabla|} f^\om$. To obtain this estimate, we interpolate between two different probabilistic estimates which exploit the different advantages of the individual decompositions of the randomization.

In fact, the angular randomization with respect to a good frame allows us to employ Strichartz estimates with (almost) the same range of exponents as for the radial wave equation. Combined with the Wiener randomization, we obtain in Proposition~\ref{prop:ProbabilisticEstimateReg} a set of almost sure space-time estimates for $e^{\pm \imu t |\nabla|} f^\om$ where the loss of derivatives is significantly improved in comparison to the deterministic case.

The physical-space randomization on the other hand gives access to the dispersive decay of the wave equation. However, in contrast to the Schr{\"o}dinger case in~\cite{SSchr2021}, a direct application of the dispersive estimate for the wave equation leads to a derivative loss which is too high in order to obtain an almost sure $L^1_t L^\infty_x$ estimate of $e^{\pm \imu t |\nabla|}f^\om$ with $f \in H^s(\R^4)$ with $s < 1$. In view of the Wiener randomization, an idea to reduce this loss of regularity is to use the Klainerman-Tataru-refinement of the dispersive estimate from~\cite{KlTa1999}. A direct application of this estimate again fails due to the additional angular decomposition between the unit-scale decomposition in frequency and in physical space. We overcome this problem by employing a variant of this estimate which only requires the data to be frequency localized to thin annuli at the expense of only being valid for large times (where the meaning of large depends on the frequency), see Remark~\ref{rem:DiscussionPhysicalSpaceRand} for further discussion. In Proposition~\ref{prop:ProbabilisticEstimateDecay} we show that this variant is still enough to obtain the dispersive decay of the wave equation for $e^{\pm \imu t |\nabla|} f^\om$ almost surely with the same loss of derivatives as in the Klainerman-Tataru estimate. Interpolating between Proposition~\ref{prop:ProbabilisticEstimateReg} and Proposition~\ref{prop:ProbabilisticEstimateDecay}, we then derive the $L^1_t L^\infty_x$-estimate in Proposition~\ref{prop:ProbabilisticL1LinftyEstiamte}.

Finally, we note that the indicated question of compatibility of the different decompositions in the randomization is one of the main technical challenges in the proofs of the probabilistic estimates. To be more precise, we have to make sure that the advantages of one of the decompositions in the randomization is not weakened by the other decompositions in the sense that we need to spend decay or regularity in order to sum up the pieces of the latter. 

The rest of the paper is organized as follows. In Section~\ref{sec:NotPrelim} we introduce some notation and collect several deterministic estimates needed in the following. In particular, we provide in Lemma~\ref{lem:ImprovedDispersiveEstimate} the variant of the improved dispersive estimate mentioned above. Section~\ref{sec:ProbabilisticEstimates} is entirely devoted to the derivation of the probabilistic estimates for the linear flow of the wave equation. In Section~\ref{sec:AlmostSureScattering} we then show how the $L^1_t L^\infty_x$-estimate implies Theorem~\ref{thm:AlmostSureScattering}.

\section{Notation and preliminaries}
\label{sec:NotPrelim}

In this section we fix some notation and gather several deterministic estimates which will be used later on.

Throughout let $d \geq 2$ in this section. We write $A \lesssim B$ if there is a constant $C > 0$ such that $A \leq C B$ and $A \sim B$ if $A \lesssim B$ and $B \lesssim A$.

\subsection*{Symbols and multipliers} Fix an even function $\eta_0 \in C_c^\infty(\R)$ such that $0 \leq \eta_0 \leq 1$, $\eta_0(x) = 1$ for $|x| \leq \frac{5}{4}$ and $\eta_0(x) = 0$ for $|x| \geq \frac{8}{5}$. For each dyadic number $N \in 2^\Z$ we define the symbols
\begin{align*}
	\chi_N(\xi) &= \eta_0(|\xi|/N) - \eta_0(2|\xi|/N), \\
	\chi_{\leq N}(\xi) &= \eta_0(|\xi|/N), \hspace{8em} \chi_{>N}(\xi) = 1 - \eta_0(|\xi|/N)
\end{align*}
on $\R^d$ and the standard Littlewood-Paley projectors as the corresponding Fourier multipliers, i.e.,
\begin{align*}
	P_N f = \cF^{-1}(\chi_N \hat{f}), \qquad P_{\leq N} f = \cF^{-1}(\chi_{\leq N} \hat{f}), \qquad P_{>N} f = \cF^{-1}(\chi_{>N} \hat{f}),
\end{align*}
where $\hat{f} = \cF f$ denotes the Fourier transform of $f$. We also employ the thickened Littlewood-Paley operators 
\begin{align*}
	\tilde{P}_N = \sum_{|\log_2 (M/N)| \leq 4} P_M
\end{align*}
for all $N \in 2^\Z$.

We further need localizers to annuli with fixed width. Let $\kappa \geq 1$. We fix a non-negative function $\zeta_\kappa \in C_c^\infty(\R)$ such that $\zeta_\kappa(s) = 1$ for $|s| \leq \frac{2}{3}\kappa$ and $\zeta_\kappa(s) = 0$ for $|s| > \frac{3}{4} \kappa$. We set
\begin{equation}
	\label{eq:DefZetaKappaLambda}
	\zeta_{\kappa}^0(s) = \frac{\zeta_\kappa(s)}{\sum_{\nu \in \Z} \zeta_\kappa(s - \nu)}, \qquad 
	\zeta_{\kappa,\lambda}(\xi) = \zeta_{\kappa}^0(|\xi|-\lambda) = \frac{\zeta_\kappa(|\xi| - \lambda)}{\sum_{\nu \in \Z} \zeta_\kappa(|\xi| - \nu)}
\end{equation}
for all $s \in \R$, $\xi \in \R^d$, and $\lambda \in \N$. Note that $\zeta_\kappa^0(|\xi|-\nu) = 0$ for all $\nu \leq 0$ if $|\xi| \geq \kappa$. Consequently, $\sum_{\lambda \in \N} \zeta_{\kappa,\lambda}(\xi) = 1$ for all $\xi \in \R^d$ with $|\xi| \geq \kappa$. We also have $\supp \zeta_{\kappa,\lambda} \subseteq \{\xi \in \R^d \colon -\frac{3}{4} \kappa + \lambda \leq |\xi| \leq \lambda + \frac{3}{4} \kappa\}$ for all $\lambda \in \N$. We define the corresponding multipliers by
\begin{equation}
	\label{eq:DefQkappalambda}
	Q_{\kappa, \lambda} f = \cF^{-1}(\zeta_{\kappa,\lambda} \hat{f})
\end{equation}
for all $\lambda \in \N$. Finally, we fix a non-negative function $\tilde{\zeta}_\kappa \in C^\infty_c(\R)$ such that $\tilde{\zeta}_\kappa(s) = 1$ for $|s| \leq \frac{3}{4} \kappa$ and $\tilde{\zeta}_\kappa(s) = 0$ for $|s| >  \kappa$. We set
\begin{equation}
	\label{eq:Deftildezeta1mu}
	\tilde{\zeta}_{\kappa,\lambda}(\xi) = \tilde{\zeta}_{\kappa}(|\xi| - \lambda)
\end{equation}
for all $\xi \in \R^d$ and $\lambda \in \N$. Note that $\tilde{\zeta}_{\kappa,\lambda} = 1$ on the support of $\zeta_{\kappa,\lambda}$ and $\supp \tilde{\zeta}_{\kappa,\lambda} \subseteq \{\xi \in \R^d \colon -\kappa + \lambda \leq |\xi| \leq \lambda + \kappa\}$ for all $\lambda \in \N$.

\subsection*{Function spaces}
Let $p,\mu \in [1,\infty]$. We set $\cL^p(0,\infty) = L^p((0,\infty), r^{d-1} \dd r)$ and define the function space $\cL^p L^\mu(\R^d)$, anisotropic in the radial and the angular variable, by the norm
\begin{align*}
	\|f\|_{\cL_r^p L^\mu_\theta} = \Big(\int_0^\infty \Big(\int_{S^{d-1}} |f(r \theta)|^{\mu} \dd \theta \Big)^{\frac{p}{\mu}} r^{d-1} \dd r \Big)^{\frac{1}{p}}
\end{align*}
with the usual adaptions if $p = \infty$ or $\mu = \infty$.

Let $I \subseteq \R$ be an interval and $q \in [1,\infty]$. For a function space $X$ with norm $\| \cdot \|_X$ we set
\begin{align*}
	\|f\|_{L^q_t X} = \Big(\int_I \|f(t)\|_X^q \dd t\Big)^{\frac{1}{q}}
\end{align*}
for the norm of $L^q(I,X)$, with the usual adaption in the case $q = \infty$. We do not specify the interval $I$ if it is clear from the context but we write $\|f\|_{L^q_I X}$ if we want to emphasize the underlying time interval.

Let $s \in \R$. We use the standard homogeneous Besov space $\dot{B}^s_{p,2}(\R^d)$ and the Besov-type spaces $\dot{B}^s_{(p,\mu),2}(\R^d)$ and $\dot{B}^s_{q,(p,\mu),2}(I \times \R^d)$ defined by the norms
\begin{align*}
	\|f\|_{\dot{B}^s_{p,2}} &= \Big(\sum_{N \in 2^\Z} N^{2s} \|P_N f\|_{L^p}^2 \Big)^{\frac{1}{2}}, \qquad
	\|f\|_{\dot{B}^s_{(p,\mu),2}} = \Big(\sum_{N \in 2^\Z} N^{2 s} \|P_N f\|_{\cL^p_r L^\mu_\theta}^2 \Big)^{\frac{1}{2}}, \\
	\|g\|_{\dot{B}^s_{q,(p,\mu),2}} &= \Big(\sum_{N \in 2^\Z} N^{2 s} \|P_N g\|_{L^q_t \cL^p_r L^\mu_\theta}^2 \Big)^{\frac{1}{2}}.
\end{align*}
Finally, we write $\dot{B}^s_{q,p,2}(I \times \R^d) = \dot{B}^s_{q,(p,p),2}(I \times \R^d)$ in the case $\mu = p$.

\subsection*{Deterministic estimates}

	We define the operators $T^{\frac{d + 2k - 2}{2}}_1$ by
	\begin{equation}
		\label{eq:DefTk1}
		T^{\frac{d + 2k - 2}{2}}_1(h)(t,r) = r^{-\frac{d-2}{2}} \int_0^\infty e^{\imu t \rho} J_{\frac{d+2k-2}{2}}(r \rho) \chi_{2^0}(\rho) h(\rho) \rho^{\frac{d}{2}} \dd \rho 
	\end{equation}
	for $t \in \R$ and $r > 0$ for all $k \in \N_0$ and $h \in L^2(0,\infty)$. In view of~\eqref{eq:RepresentationgMi}, the operators $T^{\frac{d + 2k - 2}{2}}_1$ naturally appear when we estimate the linear (half)-wave flow with randomized data.
	Although they are not explicitly mentioned in~\cite{Sterbenz2005}, the asymptotic properties of the operators $T_1^{\frac{d+2k-2}{2}}$ have been studied in~\cite{Sterbenz2005} (cf. (4.2)-(4.9) in~\cite{Sterbenz2005}). Proposition~4.1 in~\cite{Sterbenz2005} implies the following lemma.
\begin{lem}
	\label{lem:EstimateTkp11}
	Let $q,p \in [2,\infty]$ such that
	\begin{align*}
		\frac{1}{q} + \frac{d-1}{p} < \frac{d-1}{2} \qquad \text{or} \qquad (q,p)=(\infty,2).
	\end{align*}
	Then there is a constant $C > 0$ such that
	\begin{align*}
		\|T^{\frac{d+2k-2}{2}}_1(h)\|_{L^q_t \cL^p_r} \leq C \|h\|_{L^2}
	\end{align*}
	for all $k \in \N_0$ and $h \in L^2(0,\infty)$.
\end{lem}
Lemma~\ref{lem:EstimateTkp11} follows in the same way from~\cite[Proposition~4.1]{Sterbenz2005} as \cite[Theorem~1.5]{Sterbenz2005}. The proof is also implicitly contained in the proof of~\cite[Proposition~3.6]{S2020} in the case $d = 3$, see (3.8) to (3.11) and the following two estimates in~\cite{S2020} with $\hat{c}^{m}_{k,l}$ replaced by $h$, and works without changes in general dimension.

We next provide the variant of the Klainerman-Tataru improved dispersive estimate for the wave equation which we need in order to derive improved decay properties for linear solutions with randomized data without losing too many derivatives. The original version states that if the initial data is localized to a small cube in frequency space, one gains the side length of this cube as a factor in the dispersive estimate, see~(A.66) in~\cite{KlTa1999}. For our purposes it is crucial that the initial data only needs to be localized on thin annuli in frequency space. However, the arguments in the proof of~\cite[(A.66)]{KlTa1999} basically show that for large times localization in the radial variable is sufficient to obtain the improvement, see also the proof of~\cite[Lemma~3.1]{BH2017}. We provide an explicit proof here.
\begin{lem}
	\label{lem:ImprovedDispersiveEstimate}
	Let $c > 0$ and $\kappa \geq 1$. There is a constant $C > 0$ such that
	\begin{align*}
		\|e^{\pm\imu t |\nabla|} Q_{\kappa, \lambda} f \|_{L^p_x} \leq C \Big(\frac{|t|}{\lambda}\Big)^{-\frac{d-1}{2}(1-\frac{2}{p})} \|Q_{\kappa,\lambda} f\|_{L^{p'}_x}
	\end{align*}
	for all $|t| \geq c \lambda$, $\lambda \in \N$ with $\lambda \geq \kappa + 1$, and $p \in [2,\infty]$.
\end{lem}

\begin{proof}
	By time reversal symmetry, it is enough to prove the assertion for $e^{\imu t |\nabla|}$. Moreover, we only show the estimate for positive times as the proof for negative times is analogous.
	
	By interpolation with
	\begin{align*}
		\|e^{\imu t |\nabla|} Q_{\kappa,\lambda} f \|_{L^2_x} = \|Q_{\kappa,\lambda} f \|_{L^2_x},
	\end{align*}
	it is enough to prove the assertion for $p = \infty$. We next note that
	\begin{align*}
		e^{\imu t |\nabla|} Q_{\kappa,\lambda} f = K_{\kappa,\lambda}(t) \ast Q_{\kappa,\lambda} f,
	\end{align*}
	where the convolution kernel is given by
	\begin{align*}
		K_{\kappa,\lambda}(t,x) = \int_{\R^d} e^{\imu t |\xi| + \imu x \xi} \tilde{\zeta}_{\kappa,\lambda}(\xi) \dd \xi
	\end{align*}
	and $\tilde{\zeta}_{\kappa,\lambda}$ was defined in~\eqref{eq:Deftildezeta1mu}. By Young's inequality it is thus enough to prove 
	\begin{align*}
		\|K_{\kappa,\lambda}(t)\|_{L^\infty_x} \lesssim \Big(\frac{t}{\lambda}\Big)^{-\frac{d-1}{2}}
	\end{align*}
	for all $t \geq c \lambda$ and $\lambda \geq \kappa + 1$. Rescaling yields
	\begin{align}
	\label{eq:HalfWaveKernelRescaled}
		K_{\kappa,\lambda}(t,x) = \lambda^d K_{\frac{\kappa}{\lambda},1}(\lambda t, \lambda x)
	\end{align}
	for all $x \in \R^d$ and $t \in \R$, where
	\begin{align*}
		K_{\frac{\kappa}{\lambda},1}(s, y) = \int_{\R^d} e^{\imu s |\xi| + \imu y \xi} \tilde{\zeta}_{\frac{\kappa}{\lambda},1}(\xi) \dd \xi
	\end{align*}
	and $\tilde{\zeta}_{\frac{\kappa}{\lambda},1}(\xi) = \tilde{\zeta}_{\kappa,\lambda}(\lambda \xi)$ for all $\xi \in \R^d$. We then have
	\begin{align}	 
	\label{eq:SuppZetaRescaled}
	 \supp \tilde{\zeta}_{\frac{\kappa}{\lambda},1} \subseteq \Big\{\xi \in \R^d \colon -\frac{\kappa}{\lambda} + 1 \leq |\xi| \leq 1 + \frac{\kappa}{\lambda}\Big\} \subseteq \Big\{\xi \in \R^d \colon \frac{1}{1 + \kappa} \leq |\xi| \leq 2\Big\}
	 \end{align}
	 and thus
	\begin{align}
	\label{eq:EstimatesTildeZetaRescaled}
		|\partial^\alpha \tilde{\zeta}_{\frac{\kappa}{\lambda},1}(\xi)| \leq C_\alpha \lambda^{|\alpha|}, \qquad |\supp \tilde{\zeta}_{\frac{\kappa}{\lambda},1}| \lesssim \frac{\kappa}{\lambda}
	\end{align}
	for all $\xi \in \R^d$, $\alpha \in \N_0^d$, and $\lambda \geq \kappa + 1$, where the constants are independent of $\lambda$. We next show
	\begin{align}
	\label{eq:EstimateRescaledKernel}
		\|K_{\frac{\kappa}{\lambda},1}(s)\|_{L^\infty_y} \lesssim \frac{1}{\lambda} s^{-\frac{d-1}{2}} + \frac{1}{\lambda} \Big(\frac{s}{\lambda}\Big)^{-(d-1)}
	\end{align}
	for all $s > 0$ and $\lambda \geq \kappa + 1$ as this estimate combined with~\eqref{eq:HalfWaveKernelRescaled} implies
	\begin{align*}
		\|K_{\kappa,\lambda}(t,x)\|_{L^\infty_x} &= \| \lambda^d K_{\frac{\kappa}{\lambda},1}(\lambda t, \lambda x)\|_{L^\infty_x} \\
		&\lesssim \lambda^{d-1} (\lambda t)^{-\frac{d-1}{2}} + \lambda^{(d-1)} t^{-(d-1)} \lesssim \Big(\frac{t}{\lambda}\Big)^{-\frac{d-1}{2}}
	\end{align*}
	for all $t \geq c \lambda$ and $\lambda \geq \kappa + 1$. It thus remains to prove~\eqref{eq:EstimateRescaledKernel}.
	
	Let $s > 0$. In the case $|y| \leq \frac{1}{2} s$ the phase is non-stationary and we employ the usual integration by parts argument. Combined with~\eqref{eq:EstimatesTildeZetaRescaled} and exploiting that with each integration by parts, at most one derivative falls onto $\tilde{\zeta}_{\frac{\kappa}{\lambda},1}$, we obtain
	\begin{equation}
	\label{eq:EstKernelStat}
		|K_{\frac{\kappa}{\lambda},1}(s,y)| \lesssim_N \frac{1}{\lambda} \Big(\frac{s}{\lambda}\Big)^{-N}
	\end{equation}
	 for all $\lambda \geq \kappa + 1$ and every $N \in \N$, where we gain the factor $\frac{1}{\lambda}$ because of the size of the support of $\tilde{\zeta}_{\kappa,\lambda}$.
	 
	 In the case $|y| \geq \frac{1}{2} s$, we use spherical coordinates. After rotation, we can assume that $y = |y| e_d$, where $e_d = (0,\ldots,0,1)$ denotes the $d$-th unit vector. We then get
	 \begin{align}
	 \label{eq:KernelInSphericalCoord}
	 	K_{\frac{\kappa}{\lambda},1}(s,y) &= \int_0^\infty \int_0^\pi \int_{S^{d-2}} e^{\imu s \rho + \imu |y| \rho \cos(\theta)} \tilde{\zeta}_{\frac{\kappa}{\lambda},1}(\rho \, e_d) \rho^{d-1} \sin^{d-2}(\theta) \dd \om \dd \theta \dd \rho \nonumber\\
	 	&= \om_{d-2} \int_0^\infty \Big(\int_0^\pi e^{\imu |y| \rho \cos(\theta)} \sin^{d-2}(\theta) \dd \theta\Big) \, e^{\imu s \rho} \tilde{\zeta}_{\frac{\kappa}{\lambda},1}(\rho \,e_d) \rho^{d-1} \dd \rho,
	 \end{align}
	 where $\om_{d-2} = 2 \pi^{\frac{d-1}{2}}/\Gamma(\frac{d-1}{2})$ denotes the surface area of $S^{d-2}$ and we used that $\tilde{\zeta}_{\frac{\kappa}{\lambda},1}$ is radially symmetric.
	 If the dimension $d$ is even, we integrate by parts $\frac{d-2}{2}$ times with respect to $\theta$, which yields
	 \begin{align*}
	 	&\int_0^\pi e^{\imu |y| \rho \cos(\theta)} \sin^{d-2}(\theta) \dd \theta 
	 	= \frac{1}{(-\imu |y| \rho)^{\frac{d-2}{2}}} \int_0^\pi \partial_\theta^{\frac{d-2}{2}} e^{\imu |y| \rho \cos(\theta)} \sin^{\frac{d-2}{2}}(\theta) \dd \theta \\
	 	&= \frac{1}{(\imu |y| \rho)^{\frac{d-2}{2}}} \Big(\int_0^\pi e^{\imu |y| \rho \cos(\theta)} \cos^{\frac{d-2}{2}}(\theta) \dd \theta + \int_0^\pi e^{\imu |y| \rho \cos(\theta)} \sin^2(\theta) g(\theta) \dd \theta\Big),
	 \end{align*}
	 where $g$ is a linear combination of the functions $\sin^j(\theta) \cos^k(\theta)$ with $j,k \in \N_0$, $j+k = \frac{d-6}{2}$ (which also means that $g = 0$ in the case $d \leq 4$).
	 In the second integral on the above right-hand side we can integrate by parts again. Splitting the first integral via a smooth partition of unity of $[0,\pi]$, we can either integrate by parts again or employ the van der Corput lemma. Summing up, we obtain
	 \begin{align}
	 \label{eq:EstOscIntTheta}
	 	\Big| \int_0^\pi e^{\imu |y| \rho \cos(\theta)} \sin^{d-2}(\theta) \dd \theta  \Big| \lesssim (|y| \rho)^{-\frac{d-1}{2}} \lesssim s^{-\frac{d-1}{2}}
	 \end{align}
	 for all $\rho$ with $\rho \, e_d \in \supp \tilde{\zeta}_{\frac{\kappa}{\lambda},1}$, where we also used~\eqref{eq:SuppZetaRescaled} in the last estimate. If the dimension $d$ is odd, we can directly integrate $\frac{d-1}{2}$ times by parts and we also obtain~\eqref{eq:EstOscIntTheta} in this case. Combining the support property~\eqref{eq:SuppZetaRescaled} with~\eqref{eq:EstOscIntTheta}, we obtain from~\eqref{eq:KernelInSphericalCoord}
	 \begin{equation}
	 \label{eq:EstimateKernelNonStat}
	 	|K_{\frac{\kappa}{\lambda},1}(s,y)| \lesssim \frac{1}{\lambda} s^{-\frac{d-1}{2}}.
	 \end{equation}
	 Consequently, estimate~\eqref{eq:EstKernelStat} with $N = d-1$ and estimate~\eqref{eq:EstimateKernelNonStat} imply
	 \begin{align*}
	 	\|K_{\frac{\kappa}{\lambda},1}(s)\|_{L^\infty_y} \lesssim  \frac{1}{\lambda} s^{-\frac{d-1}{2}} + \frac{1}{\lambda} \Big(\frac{s}{\lambda}\Big)^{-(d-1)}
	 \end{align*}
	 for all $s > 0$ and $\lambda \in \N$ with $\lambda \geq \kappa + 1$, i.e.~\eqref{eq:EstimateRescaledKernel}.
\end{proof}

\section{Probabilistic estimates}
\label{sec:ProbabilisticEstimates}
In this section we derive the improved space-time estimates for the linear solutions of the wave equation with randomized data, culminating in the almost sure $L^1_t L^\infty_x$-estimate for $e^{\pm \imu t |\nabla|} f^\om$. We recall that if not specified otherwise, we consider general dimensions $d \geq 2$.

Since the Littlewood-Paley operators do not commute with the physical-space decomposition from~\eqref{eq:DefPhysicalSpacePartition}, we will also need Lemma~3.3 from~\cite{S2020} to sum up the individual pieces of the decomposition. This lemma was given in three dimensions in~\cite{S2020}, but both its statement and its proof are independent of the dimension.
\begin{lem}
	\label{lem:MismatchEstimates}
	Let $1 \leq p < \infty$, $M,N \in 2^{\N}$ with $|\log_2 \frac{M}{N}| \geq 5$, $l,l' \in \Z^d$, $D > 0$, and let $(\ph_l)_{l \in \Z^d}$ be the partition of unity from~\eqref{eq:DefPhysicalSpacePartition}. Then
	\begin{align}
		&\|\ph_l P_M (\ph_{l'} f)\|_{L^p} + \|\ph_l P_{\leq 2^0} (\ph_{l'}f)\|_{L^p} \lesssim_D \langle l - l' \rangle^{-D} \|f\|_{L^p}, \label{eq:MismatchSpace} \\
		&\|P_M(\ph_l P_N f)\|_{L^p}  \lesssim_D M^{-D} N^{-D} \|f\|_{L^p}, \label{eq:MismatchFrequencyHH} \\
		&\|P_M(\ph_l P_{\leq 2^{-5}M} f)\|_{L^p} \lesssim_D M^{-D} \|f\|_{L^p}, \label{eq:MismatchFrequencyHL}
	\end{align}
	for all $f \in L^p(\R^d)$, where the implicit constants are independent of $M,N,l$ and $l'$.
\end{lem}

The above lemma allows us to prove the following estimates, which naturally emerge in the proof of the probabilistic space-time estimates. The first part is similar to~\cite[Corollary~3.3]{SSchr2021} but we need to adapt the proof slightly in order to include the case $s < 0$. The second part of the corollary is more delicate. Since we want to apply the improved Klainerman-Tataru dispersive estimate in the form of Lemma~\ref{lem:ImprovedDispersiveEstimate}, we need to sum up the individual pieces of the decomposition for the randomization intertwined with the operators $Q_{\kappa,\lambda}$.
\begin{cor}
	\label{cor:HsEstimatePhysicalSpace}
	Let $p \in [2,\infty)$, $s \in \R$, and $(\ph_l)_{l \in \Z^d}$ be the partition of unity introduced in~\eqref{eq:DefPhysicalSpacePartition}.
	\begin{enumerate}
		\item \label{it:HsMismatchLP} We then have
				\begin{align*}
					\|\langle M \rangle^s P_{>4} P_M(\ph_l f)\|_{l^2_M l^2_l L^{p'}_x} \lesssim \|f\|_{H^s}
				\end{align*}
				for all $f \in H^s(\R^d)$.
		\item \label{it:HsMismatchLPAndAn} Let $\kappa \geq 1$ and $Q_{\kappa,\lambda}$ be the operators introduced in~\eqref{eq:DefQkappalambda} for $\lambda \in \N$. Then
				\begin{align*}
					\|\langle M \rangle^s Q_{\kappa,\lambda} P_{>4} P_M(\ph_l f)\|_{l^2_M l^2_l l^2_{\lambda \geq \kappa + 1} L^{p'}_x} \lesssim \|f\|_{H^s}
				\end{align*}
				for all $f \in H^s(\R^d)$.
	\end{enumerate}
\end{cor}

\begin{proof}
	\ref{it:HsMismatchLP} Writing $l^2_{M \geq 4}$ for the $l^2$-norm over those $M \in 2^\Z$ with $M \geq 4$, we first note that
	\begin{align}
	\label{eq:EstMismatchFirstReduction}
		\|\langle M \rangle^s P_{>4} P_M(\ph_l f)\|_{l^2_M l^2_l L^{p'}_x} \lesssim \|\langle M \rangle^s  P_M(\ph_l f)\|_{l^2_{M\geq 4} l^2_l L^{p'}_x}.
	\end{align}
	We then estimate
	\begin{align}
		\| P_M(\ph_l f)\|_{l^2_l L^{p'}_x} &\leq \| P_M(\ph_l P_{\leq \min\{1, 2^{-5}M\}} f)\|_{l^2_l L^{p'}_x} + \| P_M(\ph_l \tilde{P}_M f)\|_{l^2_l L^{p'}_x} \nonumber\\
		&\qquad + \sum_{\substack{N \in 2^\N \\|\log_2(M/N)| \geq 5}}\| P_M(\ph_l P_N f)\|_{l^2_l L^{p'}_x}. \label{eq:MismatchSplitting}
	\end{align}
	Treating the second summand on the right-hand side first and using $|\supp \ph_l| \lesssim 1$ for all $l \in \Z^d$ combined with H{\"o}lder's inequality, we get
	\begin{equation}
		\label{eq:MismatchFreqComp}
		\| P_M(\ph_l \tilde{P}_M f)\|_{l^2_l L^{p'}_x} \lesssim \| \ph_l \tilde{P}_M f\|_{l^2_l L^{p'}_x} \lesssim \|\ph_l \tilde{P}_M f\|_{l^2_l L^{2}_x} \lesssim \|\tilde{P}_M f\|_{L^2_x}.
	\end{equation}
	To estimate the sum on the right-hand side of~\eqref{eq:MismatchSplitting}, we first define the thickened localizers in  physical space $\tilde{\ph}_l = \sum_{m \in \Z^d,|m-l| \leq 2 \sqrt{d}} \ph_m$. Then $\tilde{\ph}_l = 1$ on the support of $\ph_l$ for all $l \in \Z^d$ and we infer
	\begin{align*}
		\| P_M(\ph_l P_N f)\|_{ L^{p'}_x} 
		&\leq \sum_{l' \in \Z^d} \|P_M(\ph_l P_N (\ph_{l'} \tilde{P}_N f) \|_{L^{p'}_x} \\
		&\lesssim \sum_{l' \in \Z^d} \|P_M(\ph_l P_N (\ph_{l'} \tilde{P}_N f)) \|_{L^{p'}_x}^{\frac{1}{2}} \|\ph_l P_N (\ph_{l'} \tilde{P}_N f)\|_{L^{p'}_x}^{\frac{1}{2}} \\
		&\lesssim \sum_{l' \in \Z^d} M^{-D} N^{-D} \langle l - l' \rangle^{-2d-2} \| \tilde{\ph}_{l'} \tilde{P}_N f\|_{L^{p'}_x}
	\end{align*}
	for all $N \in 2^\N$ with $|\log_2(M/N)| \geq 5$, $l \in \Z^d$ and a constant $D > |s|$, where we applied Lemma~\ref{lem:MismatchEstimates} in the last step. Since $\langle x - l \rangle \lesssim \langle l' - l\rangle$ for all $x \in \supp \tilde{\ph}_{l'}$ and $|\supp \tilde{\ph}_{l'}| \lesssim 1$ for all $l' \in \Z^d$, we thus obtain
	\begin{align*}
		\| P_M(\ph_l P_N f)\|_{L^{p'}_x} 
		&\lesssim \sum_{l' \in \Z^d} M^{-D} N^{-D} \langle l - l' \rangle^{-2d-2} \|\tilde{\ph}_{l'} \tilde{P}_N f\|_{L^2_x} \\
		&\lesssim  \sum_{l' \in \Z^d} M^{-D} N^{-D} \langle l - l' \rangle^{-d-1} \| \langle x - l \rangle^{-d-1} \tilde{P}_N f\|_{L^2_x} \\
		&\lesssim M^{-D} N^{-D} \| \langle x - l \rangle^{-d-1} \tilde{P}_N f\|_{L^2_x}.
	\end{align*}
	We next take the $l^2$-norm in $l$ and sum over $N \in 2^\N$ with $|\log_2(M/N)| \geq 5$, which yields
	\begin{align}
		\sum_{\substack{N \in 2^\N \\|\log_2(M/N)| \geq 5}}\| P_M(\ph_l P_N f)\|_{l^2_l L^{p'}_x} &\lesssim \sum_{\substack{N \in 2^\N \\|\log_2(M/N)| \geq 5}} M^{-D} N^{-D - s} N^{s}\| \tilde{P}_N f\|_{ L^{2}_x} \nonumber\\
		& \lesssim M^{-D} \|f\|_{H^s} \label{eq:EstMismatchFinalDifferentFrequencies}
	\end{align}	
	as $D - |s| > 0$.	If $M \geq 2^5$, we argue analogously for the first term on the right-hand side of~\eqref{eq:MismatchSplitting}, which yields
	\begin{equation}
		\label{eq:EstMismatchLowFrequency}
		\| P_M(\ph_l P_{\leq 1} f)\|_{l^2_l L^{p'}_x} \lesssim M^{-D} \|P_{\leq 2} f\|_{L^2_x} \lesssim M^{-D} \|f\|_{H^s}.
	\end{equation}
	If $M \in 2^\N$ with $4 \leq M \leq 2^4$, we simply estimate
	\begin{equation}
		\label{eq:EstMismatchLowLowFrequency}
		\| P_M(\ph_l P_{\leq 2^{-5}M} f)\|_{l^2_l L^{p'}_x} \lesssim \|P_{\leq  2^{-5}M} f\|_{L^2_x} \lesssim \|f\|_{H^s} \lesssim M^{-D} \|f\|_{H^s}.
	\end{equation}
	Inserting~\eqref{eq:MismatchFreqComp}, \eqref{eq:EstMismatchFinalDifferentFrequencies}, \eqref{eq:EstMismatchLowFrequency} and~\eqref{eq:EstMismatchLowLowFrequency} into~\eqref{eq:MismatchSplitting} and~\eqref{eq:EstMismatchFirstReduction}, we finally arrive at
	\begin{align*}
		\|\langle M \rangle^s P_{>4} P_M(\ph_l f)\|_{l^2_M l^2_l L^{p'}_x} 
		\lesssim \| \langle M \rangle^s \|\tilde{P}_M f\|_{L^2} + \langle M \rangle^s M^{-D} \|f\|_{H^s}\|_{l^2_{M \geq 4}} \lesssim \|f\|_{H^s}.
	\end{align*}
	
	\ref{it:HsMismatchLPAndAn} We first note that for every $\alpha \in \N_0^d$ there is a constant $C_\alpha$ such that
	\begin{align*}
		|\partial^\alpha \zeta_{\kappa,\lambda}(\xi)| \leq C_\alpha \lambda^{|\alpha|} |\xi|^{-|\alpha|}
	\end{align*}
	for all $\xi \in \R^d$ and $\lambda \geq \kappa + 1$, since 
	\begin{align*}
		\zeta_{\kappa,\lambda} = \zeta_{\kappa}^0(|\xi| - \lambda) \quad \text{and} \quad \supp \zeta_{\kappa,\lambda} \subseteq \{\xi \in \R^d \colon \lambda - \kappa \leq |\xi| \leq \lambda + \kappa \}
	\end{align*}		
	 for all $\lambda \in \N$. The Mihlin multiplier theorem thus shows that for every $q \in (1,\infty)$ we have
	 \begin{equation}
	 	\label{eq:EstMultiplierBoundQkl}
	 		\|Q_{\kappa,\lambda} f\|_{L^q} \lesssim \lambda^{\lfloor \frac{d}{2}\rfloor + 1} \|f\|_{L^q}
	 \end{equation}
	 for all $\lambda \geq \kappa + 1$, where the implicit constant is independent of $\lambda$. Using the same notation as in part~\ref{it:HsMismatchLP}, we first note that
	 \begin{equation}
	 	\label{eq:EstMismatch2FirstReduction}
	 	\|\langle M \rangle^s Q_{\kappa,\lambda} P_{>4} P_M(\ph_l f)\|_{l^2_M l^2_l l^2_{\lambda \geq \kappa + 1} L^{p'}_x} \lesssim \|\langle M \rangle^s Q_{\kappa,\lambda} P_M(\ph_l f)\|_{l^2_{M \geq 4}  l^2_{\lambda \geq \kappa + 1} l^2_l L^{p'}_x}.
	 \end{equation}
	 Proceeding as in~(i), we estimate
	 \begin{align}
	 	\label{eq:EstMismatch2Splitting}
	 	&\|\langle M \rangle^s Q_{\kappa,\lambda} P_M(\ph_l f)\|_{l^2_{\lambda \geq \kappa + 1} l^2_l L^{p'}_x} \leq \|\langle M \rangle^s Q_{\kappa,\lambda} P_M(\ph_l P_{\leq \min\{1,2^{-5}M\}} f)\|_{l^2_{\lambda \sim M} l^2_l L^{p'}_x}  \nonumber\\
	 	&\qquad + \|\langle M \rangle^s Q_{\kappa,\lambda} P_M(\ph_l \tilde{P}_M f)\|_{l^2_{\lambda \sim M} l^2_l L^{p'}_x} \nonumber \\
	 	&\qquad + \sum_{\substack{N \in 2^\N \\|\log_2(M/N)| \geq 5}}\|\langle M \rangle^s Q_{\kappa,\lambda} P_M(\ph_l P_N f)\|_{l^2_{\lambda \sim M} l^2_l L^{p'}_x}, 
	 \end{align}
	 where $l^2_{\lambda \sim M}$ means the $l^2$-norm over those $\lambda \in \N$ with $\lambda \geq \kappa + 1$ which satisfy  $\frac{M}{2} \leq \lambda + \kappa$ and $\lambda - \kappa \leq 2M$, using that $Q_{\kappa,\lambda} P_M$ is zero for all other $\lambda$. We thus obtain from~\eqref{eq:EstMultiplierBoundQkl} and~\eqref{eq:EstMismatchFinalDifferentFrequencies} with $D > \lfloor \frac{d}{2}\rfloor + 2 + |s|$ that
	 \begin{align}
	 	&\sum_{\substack{N \in 2^\N \\|\log_2(M/N)| \geq 5}}\|\langle M \rangle^s Q_{\kappa,\lambda} P_M(\ph_l P_N f)\|_{l^2_{\lambda \sim M} l^2_l L^{p'}_x} \nonumber \\
	 	&\lesssim  \sum_{\substack{N \in 2^\N \\|\log_2(M/N)| \geq 5}}  \|\langle M \rangle^s \lambda^{\lfloor \frac{d}{2}\rfloor + 1} P_M(\ph_l P_N f)\|_{l^2_{\lambda \sim M} l^2_l L^{p'}_x} \nonumber\\
	 	&\lesssim \|\langle M \rangle^s M^{\lfloor \frac{d}{2}\rfloor + 1} M^{-D} \|f\|_{H^s}\|_{l^2_{\lambda \sim M}} \lesssim M^{\lfloor \frac{d}{2} \rfloor + \frac{3}{2} + |s| -D} \|f\|_{H^s}. \label{eq:EstMismatch2FinalDifferentFrequencies}
	 \end{align}
	Employing~\eqref{eq:EstMismatchLowFrequency} respectively~\eqref{eq:EstMismatchLowLowFrequency} instead of~\eqref{eq:EstMismatchFinalDifferentFrequencies}, we obtain in the same way
	\begin{equation}
		\label{eq:EstMismatch2LowFrequency}
		\|\langle M \rangle^s Q_{\kappa,\lambda} P_M(\ph_l P_{\leq \min\{1,2^{-5}M\}} f)\|_{l^2_{\lambda \sim M} l^2_l L^{p'}_x} \lesssim M^{\lfloor \frac{d}{2} \rfloor + \frac{3}{2} + |s| -D} \|f\|_{H^s}.
	\end{equation}
	It remains to estimate the second term on the right-hand side of~\eqref{eq:EstMismatch2Splitting}. Here we first note that
	\begin{align}
		 &\| Q_{\kappa,\lambda} P_M(\ph_l \tilde{P}_M f)\|_{l^2_{\lambda \sim M}  L^{p'}_x} \nonumber \\
		 & \leq \|\ph_l Q_{\kappa,\lambda} P_M(\ph_l \tilde{P}_M f)\|_{l^2_{\lambda \sim M}  L^{p'}_x} + \sum_{l' \in \Z^d, l' \neq l} \| \ph_{l'} Q_{\kappa,\lambda} P_M(\ph_l \tilde{P}_M f)\|_{l^2_{\lambda \sim M}  L^{p'}_x} \nonumber \\
		 &\lesssim \|\ph_l Q_{\kappa,\lambda} P_M(\ph_l \tilde{P}_M f)\|_{l^2_{\lambda \sim M}  L^{2}_x} + \sum_{l' \in \Z^d, l' \neq l} \| \ph_{l'} Q_{\kappa,\lambda} P_M(\ph_l \tilde{P}_M f)\|_{l^2_{\lambda \sim M}  L^{2}_x}, \label{eq:EstMismatch2PartitionPhysicalSpace}
	\end{align}
	where we used that $(\ph_l)_{l \in \Z^d}$ is a partition of unity with $|\supp \ph_l| \lesssim 1$ for all $l \in \Z^d$. For the first term we infer 
	\begin{align}
	\label{eq:EstMismatch2FirstTermL2}
		\|\ph_l Q_{\kappa,\lambda} P_M(\ph_l \tilde{P}_M f)\|_{l^2_{\lambda \sim M}  L^{2}_x} &\lesssim \| Q_{\kappa,\lambda} P_M(\ph_l \tilde{P}_M f)\|_{ L^{2}_x l^2_{\lambda \sim M} } \lesssim  \|P_M(\ph_l \tilde{P}_M f)\|_{ L^{2}_x} \nonumber\\
		&\lesssim  \| \ph_l \tilde{P}_M f\|_{L^{2}_x},
	\end{align}
	exploiting that $(\zeta_{\kappa,\lambda})_{\lambda \in \Z}$ is a partition of unity on $\R^d$. To treat the sum on the right-hand side of~\eqref{eq:EstMismatch2PartitionPhysicalSpace}, we compute
	\begin{align*}
		\cF(\ph_{l'} Q_{\kappa,\lambda} P_M(\ph_l \tilde{P}_M f))(\xi) = \int_{\R^d} \!\hat{\ph}_{l'}(\xi - \eta) \zeta_{\kappa,\lambda}(\eta) \chi_{M}(\eta) \int_{\R^d} \! \hat{\ph}_l(\eta - \nu) \hat{g}_{l,M}(\nu) \dd \nu \dd \eta
	\end{align*}
	for all $\xi \in \R^d$, where $g_{l,M} = \tilde{\ph}_l \tilde{P}_M f$. We recall that $\tilde{\ph}_l = \sum_{m \in \Z^d, |l - m| \leq 2\sqrt{d}} \ph_m$ equals $1$ on the support of $\ph_l$. Employing the definition of the $\ph_l$, we thus obtain
	\begin{align*}
		&\cF(\ph_{l'} Q_{\kappa,\lambda} P_M(\ph_l \tilde{P}_M f))(\xi) \\
		&= \int_{\R^d}  \int_{\R^d} e^{- \imu l' (\xi - \eta)}\hat{\ph}_0(\xi - \eta) \zeta_{\kappa,\lambda}(\eta) \chi_{M}(\eta) e^{- \imu l (\eta - \nu)}\hat{\ph}_0(\eta - \nu) \hat{g}_{l,M}(\nu) \dd \nu \dd \eta \\
		&= \int_{\R^d}  \int_{\R^d} e^{- \imu l' \xi }\hat{\ph}_0(\xi - \eta) \zeta_{\kappa,\lambda}(\eta) \chi_{M}(\eta) \hat{\ph}_0(\eta - \nu) e^{- \imu (l-l') \eta} e^{\imu l \nu} \hat{g}_{l,M}(\nu) \dd \nu \dd \eta
	\end{align*}
	for all $\xi \in \R^d$. Take $n \in \N$. Applying the identity
	\begin{align*}
		e^{- \imu (l-l') \eta} = \frac{\imu (l - l')}{|l - l'|^2} \cdot \nabla_\eta e^{- \imu (l - l') \eta} 
	\end{align*}
	$n$ times and integrating by parts in $\eta$ for $l \neq l'$, we infer that there are constants $C_{\alpha_1, \ldots, \alpha_4} \in \N$ for all $\alpha_1, \ldots, \alpha_4 \in \N_0^d$ with $|\sum_{i = 1}^4 \alpha_i| = n$ such that
	\begin{align*}
		&\cF(\ph_{l'} Q_{\kappa,\lambda} P_M(\ph_l \tilde{P}_M f))(\xi) = \sum_{\substack{\alpha_1, \ldots, \alpha_4 \in \N_0^d \\ |\sum_{i = 1}^n \alpha_i| = n}} C_{\alpha_1, \ldots, \alpha_4} \frac{(-\imu)^n}{|l - l'|^{2n}} \cdot \\
		&\quad \cdot \int_{\R^d}  \int_{\R^d} e^{- \imu l' \xi }(l-l')^{\alpha_1} \partial^{\alpha_1}_\eta \hat{\ph}_0(\xi - \eta) (l-l')^{\alpha_2} \partial^{\alpha_2}\zeta_{\kappa,\lambda}(\eta) (l-l')^{\alpha_3} \partial^{\alpha_3}\chi_{M}(\eta) \cdot \\
		&\hspace{6em} \cdot (l-l')^{\alpha_4} \partial^{\alpha_4}_\eta\hat{\ph}_0(\eta - \nu) e^{- \imu (l-l') \eta} e^{\imu l \nu} \hat{g}_{l,M}(\nu) \dd \nu \dd \eta
	\end{align*}
	for all $\xi \in \R^d$. Choosing $n = d+1$, we thus deduce via Plancherel's theorem and Young's inequality that
	\begin{align*}
		&\sum_{l' \in \Z^d, l' \neq l} \| \ph_{l'} Q_{\kappa,\lambda} P_M(\ph_l \tilde{P}_M f)\|_{l^2_{\lambda \sim M}  L^{2}_x} \nonumber\\
		&\lesssim \sum_{l' \in \Z^d, l' \neq l} \sum_{\substack{\alpha_1, \ldots, \alpha_4 \in \N_0^d \\ |\sum_{i = 1}^n \alpha_i| = n}} \frac{1}{|l - l'|^{n}} \Big\|\int_{\R^d}  \int_{\R^d} |\partial^{\alpha_1} \hat{\ph}_0(\xi - \eta)|\cdot |\partial^{\alpha_2}\zeta_{\kappa,\lambda}(\eta)|\cdot | \partial^{\alpha_3}\chi_{M}(\eta)|\cdot \nonumber \\
		&\hspace{16em} \cdot |\partial^{\alpha_4}\hat{\ph}_0(\eta - \nu)||\hat{g}_{l,M}(\nu)| \dd \nu \dd \eta \Big\|_{l^2_{\lambda \sim M}  L^{2}_{\xi}} \nonumber \\
		&\lesssim  \sum_{\substack{\alpha_1, \ldots, \alpha_4 \in \N_0^d \\ |\sum_{i = 1}^n \alpha_i| = n}} \| \partial^{\alpha_1} \hat{\ph}_0\|_{L^1} \Big\|\int_{\R^d}|\partial^{\alpha_2}\zeta_{\kappa,\lambda}(\eta)|\cdot | \partial^{\alpha_3}\chi_{M}(\eta)|\cdot \nonumber \\
		&\hspace{16em} \cdot |\partial^{\alpha_4}\hat{\ph}_0(\eta - \nu)| \cdot |\hat{g}_{l,M}(\nu)| \dd \nu \Big\|_{l^2_{\lambda \sim M}  L^{2}_{\eta}}.
	\end{align*}
	We next observe that for every $\alpha_2 \in \N_0^d$ there is a constant $C_{\alpha_2}$ such that
	\begin{align*}
		|\partial^{\alpha_2} \zeta_{\kappa,\lambda}(\eta)| \leq C_{\alpha_2}
	\end{align*}
	for all $\eta \in \R^d$ and $\lambda \geq \kappa + 1$ since $\zeta_{\kappa,\lambda}(\eta) = \zeta_{\kappa}^0(|\eta| - \lambda)$ for all $\lambda \in \N$. Moreover, for every $\alpha_3 \in \N_0^d$ there is a constant $C_{\alpha_3}$ such that $\|\partial^{\alpha_3} \chi_M\|_{L^\infty} \leq C_{\alpha_3}$ for all $M \geq 4$.
	Since $\hat{\ph}_0$ is a Schwartz function and only finitely many $\zeta_{\kappa,\lambda}$ have overlapping support, we therefore get
	\begin{align}
		&\sum_{l' \in \Z^d, l' \neq l} \| \ph_{l'} Q_{\kappa,\lambda} P_M(\ph_l \tilde{P}_M f)\|_{l^2_{\lambda \sim M}  L^{2}_x} \nonumber\\
		&\lesssim  \sum_{\substack{\alpha_1, \ldots, \alpha_4 \in \N_0^d \\ |\sum_{i = 1}^n \alpha_i| = n}}  \Big\|\int_{\R^d}\|\partial^{\alpha_2}\zeta_{\kappa,\lambda}(\eta)\|_{l^2_{\lambda \sim M}} | \partial^{\alpha_3}\chi_{M}(\eta)| \!\cdot \! |\partial^{\alpha_4}\hat{\ph}_0(\eta - \nu)|\! \cdot \!|\hat{g}_{l,M}(\nu)| \dd \nu \dd \eta \Big\|_{L^{2}_{\eta}} \nonumber\\
		&\lesssim  \sum_{\substack{\alpha_1, \ldots, \alpha_4 \in \N_0^d \\ |\sum_{i = 1}^n \alpha_i| = n}}  \Big\|\int_{\R^d} |\partial^{\alpha_4}\hat{\ph}_0(\eta - \nu)| \!\cdot \!|\hat{g}_{l,M}(\nu)| \dd \nu  \Big\|_{L^{2}_{\eta}} \nonumber \\
		&\lesssim  \sum_{\substack{\alpha_1, \ldots, \alpha_4 \in \N_0^d \\ |\sum_{i = 1}^n \alpha_i| = n}}  \|\partial^{\alpha_4}\hat{\ph}_0\|_{L^1} \|\hat{g}_{l,M}\|_{L^2} \lesssim \|g_{l,M}\|_{L^2}. \label{eq:EstMismatch2SecondTermL2}
	\end{align}
	Recalling that $g_{l,M} = \tilde{\ph}_l \tilde{P}_M f$ and inserting~\eqref{eq:EstMismatch2SecondTermL2} and~\eqref{eq:EstMismatch2FirstTermL2} into~\eqref{eq:EstMismatch2PartitionPhysicalSpace}, we obtain
	\begin{align*}
		\| Q_{\kappa,\lambda} P_M(\ph_l \tilde{P}_M f)\|_{l^2_{\lambda \sim M}  L^{p'}_x} \lesssim  \|\tilde{\ph}_l \tilde{P}_M f\|_{L^2_x}.
	\end{align*}
	Consequently,
	\begin{align}
		 \|\langle M \rangle^s Q_{\kappa,\lambda} P_M(\ph_l \tilde{P}_M f)\|_{l^2_{\lambda \sim M} l^2_l L^{p'}_x} &\lesssim  \|\langle M \rangle^s \tilde{\ph}_l \tilde{P}_M f\|_{l^2_l L^{2}_x} \lesssim \|\langle M \rangle^s \tilde{P}_M f\|_{L^{2}_x}. \label{eq:EstMismatch2CompFrequencyFinal}
	\end{align}
	Finally, we insert~\eqref{eq:EstMismatch2CompFrequencyFinal}, \eqref{eq:EstMismatch2LowFrequency}, and~\eqref{eq:EstMismatch2FinalDifferentFrequencies} into~\eqref{eq:EstMismatch2Splitting} and take the $l^2_{M \geq 4}$-norm to conclude
	\begin{align*}
		\|\langle M \rangle^s Q_{\kappa,\lambda} P_M(\ph_l f)\|_{l^2_{M \geq 4}  l^2_{\lambda \geq \kappa + 1} l^2_l L^{p'}_x} \lesssim \|f\|_{H^s}.
	\end{align*}
	In view of~\eqref{eq:EstMismatch2FirstReduction}, this proves the assertion.
\end{proof}

We finally turn to the probabilistic space-time estimates for linear solutions of the wave equation with randomized data. We start by fixing some notation which we will use in the proofs of both Proposition~\ref{prop:ProbabilisticEstimateReg} and Proposition~\ref{prop:ProbabilisticEstimateDecay}.

Recall from the definition of the randomization $f^\om$ in~\eqref{eq:DefRandomization} that
	\begin{equation}
	\label{eq:DeffMom}
		f^\om = \sum_{M \in 2^\Z} f^{M,\om},
	\end{equation}
	where
	\begin{equation}
	\label{eq:SplittingfMom}
		f^{M,\om} = \sum_{i,j \in \Z^d} \sum_{k = 0}^\infty \sum_{l = 1}^{N_k} X^M_{i,j,k,l}(\om) P_j f^M_{i,k,l}
	\end{equation}
	and
	\begin{equation}
	\label{eq:DeffMikl}
		f^M_{i,k,l}(r \theta) = a_k M^{-\frac{d-2}{2}} r^{-\frac{d-2}{2}} b_{k,l}(\theta) \int_0^\infty \hat{c}^{M,i}_{k,l}(\rho) J_{\frac{d+2k-2}{2}}(M r \rho) \rho^{\frac{d}{2}} \dd \rho.
	\end{equation}
	Moreover, we know from~\eqref{eq:DecompSphHarm} that
	\begin{equation}
	\label{eq:SumfMikl}
		P_M(\ph_i f) = \sum_{k = 0}^\infty \sum_{l = 1}^{N_k} f^M_{i,k,l}.
	\end{equation}

The next proposition provides us with a set of Strichartz estimates for solutions of the linear (half)-wave equation with randomized data which is very flexible in applications due to its wide range of admissible exponents. This range is restricted on one side by the range of exponents of deterministic Strichartz estimates for the radial wave equation but allows for arbitrary high integrability in the radial and the angular variable. We note that in this work we do not need the additional flexibility in the angular variable, i.e. we only apply the proposition with $\mu = p$, but as there is no difference in the proof we provide the more general statement.

In the proof we exploit the randomization in frequency space and in the angular variable. The randomization in the angular variable gives access to Strichartz estimates in the range of the radial wave equation while the randomization in frequency space allows us to move from higher to lower spatial integrability without losing derivatives. However, we cannot simply apply the unit-scale Bernstein inequality as it is typically done in the case of the Wiener randomization since we would lose summability due to the different decompositions. We overcome this problem by applying the square function estimate from~\cite[Lemma~2.3]{SSchr2021} instead. Moreover, it is important to note that the additional decomposition in physical space does not impair the advantages of the randomization in frequency space and in the angular variable.
\begin{prop}
	\label{prop:ProbabilisticEstimateReg}
	Let $d \geq 2$. Let $s \in \R$ and $q, p_0 \in [2,\infty)$ such that $(q,p_0)$ satisfies
	\begin{align*}
		\frac{1}{q} + \frac{d-1}{p_0} < \frac{d-1}{2}.
	\end{align*}
	Take $f \in H^s(\R^d)$ and let $f^\om$ be its randomization from~\eqref{eq:DefRandomization}. Then there exists a constant $C > 0$ such that
	\begin{align*}
		 \|e^{\pm \imu t |\nabla|} P_{> 4} f^\om\|_{L^\beta_\om \dot{B}^{s + \frac{1}{q} + \frac{d}{p_0} - \frac{d}{2}}_{q,(p,\mu),2}} \leq C \sqrt{\beta} \|f\|_{H^s}
	\end{align*}
	for all $\beta \in [1,\infty)$, $p \in [p_0, \infty)$, and $\mu \in [2,\infty)$.
\end{prop}

\begin{proof}
	It is enough to prove the assertion for $\beta \geq \max\{p,q,\mu\}$ as $\Omega$ is a probability space. Moreover, it is sufficient to prove the assertion for $e^{\imu t |\nabla|}$.
	
	Note that by definition and Minkowski's inequality, we have
	\begin{align*}
		 &\|e^{\imu t |\nabla|} P_{> 4} f^\om\|_{L^\beta_\om \dot{B}^{s + \frac{1}{q} + \frac{d}{p_0} - \frac{d}{2}}_{q,(p,\mu),2}} \\
		 &\sim  \Big\| \Big( \sum_{M \in 2^{\Z}}  M^{2(s + \frac{1}{q} + \frac{d}{p_0} - \frac{d}{2})} \| P_M e^{\imu t |\nabla|} P_{>4} f^{\omega} \|_{L^q_t \cL^p_r L^\mu_\theta}^2 \Big)^{\frac{1}{2}} \Big\|_{L^\beta_\omega } \\
		&\lesssim \Big(\sum_{M \in 2^{\Z}}  M^{2(s +\frac{1}{q} + \frac{d}{p_0} - \frac{d}{2})} \|P_M e^{\imu t |\nabla|} P_{>4} f^{\omega} \|_{L^\beta_\omega L^q_t \cL^p_r L^\mu_\theta}^2 \Big)^{\frac{1}{2}}.
	\end{align*}
	 In view of Corollary~\ref{cor:HsEstimatePhysicalSpace}~\ref{it:HsMismatchLP}, it thus suffices to show
	\begin{equation}
		\label{eq:ProbEstRegRed1}
		 M^{\frac{1}{q} + \frac{d}{p_0} - \frac{d}{2}} \|P_M e^{\imu t |\nabla|} f^{\omega} \|_{L^\beta_\omega L^q_t \cL^p_r L^\mu_\theta} \lesssim \sqrt{\beta} \|\tilde{P}_M  (\ph_i f)\|_{l^2_i L^2_x}
	\end{equation}
	for every $M \in 2^{\N}$ with $M \geq 4$, where the implicit constant is independent of $M$. Using the notation from~\eqref{eq:DeffMom} to~\eqref{eq:DeffMikl} and recalling that $f^{M,\om}$ has frequency support in the annulus $\{\frac{M}{2} < |\xi| < 2M\}$ for every $M \in 2^\Z$, estimate~\eqref{eq:ProbEstRegRed1} further reduces to 
	\begin{equation}
		\label{eq:ProbEstRegRed2}
		 M^{\frac{1}{q} + \frac{d}{p_0} - \frac{d}{2}} \|e^{\imu t |\nabla|} f^{M,\omega} \|_{L^\beta_\omega L^q_t \cL^p_r L^\mu_\theta} \lesssim \sqrt{\beta} \|P_M (\ph_i f)\|_{l^2_i L^2_x}
	\end{equation}
	for all $M \in 2^{\N}$. We next apply Minkowski's inequality and the Khintchine inequality from Lemma~3.1 in~\cite{BT2008I} to infer
	\begin{align}
		\|e^{\imu t |\nabla|}  f^{M,\omega} \|_{L^\beta_\omega L^q_t \cL^p_r L^\mu_\theta} &\lesssim  \sqrt{\beta} \|e^{\imu t |\nabla|}  P_j f^M_{i,k,l} \|_{L^q_t \cL^p_r L^\mu_\theta l^2_{i,j,k,l}} \nonumber\\ 
		&\lesssim \sqrt{\beta} \|e^{\imu t |\nabla|} P_j f^M_{i,k,l} \|_{l^2_{i,k,l} L^q_t \cL^p_r L^\mu_\theta l^2_{j}}. \label{eq:EstProbRegLargeDev}
	\end{align}
	Lemma~2.3 from~\cite{SSchr2021} with radial integrability parameter $p_0$ on the right-hand side yields
	\begin{equation}
		\|P_j e^{\imu t |\nabla|}  f^M_{i,k,l} \|_{l^2_{i,k,l} L^q_t \cL^p_r L^\mu_\theta l^2_{j}} \lesssim \|e^{\imu t |\nabla|} f^M_{i,k,l} \|_{l^2_{i,k,l} L^q_t \cL^{p_0}_r L^\mu_\theta}. \label{eq:EstProbRegSquareFunction}
	\end{equation}
	Combining~\eqref{eq:ProbEstRegRed2}, \eqref{eq:EstProbRegLargeDev}, and~\eqref{eq:EstProbRegSquareFunction}, we infer that~\eqref{eq:ProbEstRegRed1} follows from
	\begin{equation}
		M^{\frac{1}{q} + \frac{d}{p_0} - \frac{d}{2}} \|e^{\imu t |\nabla|}  f^M_{i,k,l} \|_{l^2_{i,k,l} L^q_t \cL^{p_0}_r L^\mu_\theta} \lesssim  \|P_M (\ph_i f)\|_{l^2_i L^2_x} \label{eq:EstProbRegRed3}
	\end{equation}
	for every $M \in 2^\N$. To prove the latter estimate, we recall the notation $g^M_i = (P_M (\ph_i f))(M^{-1} \cdot)$ from~\eqref{eq:DefgM}. We further rescale $f^M_{i,k,l}$ to unit frequency by setting $g^M_{i,k,l} = f^{M}_{i,k,l}(M^{-1} \cdot)$. Consequently, we have the representation
	\begin{equation}
	\label{eq:RepresentaiongMikl}
		g^M_{i,k,l}(r \theta) = a_k r^{-\frac{d-2}{2}} b_{k,l}(\theta) \int_0^\infty c^{M,i}_{k,l}(\rho) J_{\frac{d+2k-2}{2}}(r \rho) \rho^{\frac{d}{2}} \dd \rho.
	\end{equation}
	We also note that~\eqref{eq:EstProbRegRed3} is equivalent to
	\begin{equation}
	\label{eq:EstProbRegUnitFreq}
		 \|e^{\imu t |\nabla|}  g^M_{i,k,l} \|_{l^2_{i,k,l} L^q_t \cL^{p_0}_r L^\mu_\theta} \lesssim  \|g^M_{i}\|_{l^2_i L^2_x}
	\end{equation}
	by scaling. Since $g^M_{i,k,l}$ has unit frequency, it suffices to show
	\begin{equation}
	\label{eq:EstProbRegPrepForT}
		\|P_{2^0} e^{\imu t |\nabla|}  g^M_{i,k,l} \|_{l^2_{i,k,l} L^q_t \cL^{p_0}_r L^\mu_\theta} \lesssim  \|g^M_{i}\|_{l^2_i L^2_x}
	\end{equation}
	in order to conclude~\eqref{eq:EstProbRegUnitFreq}. By Theorem~3.10 from~\cite{SW71} and~\eqref{eq:RepresentaiongMikl}, we obtain
	\begin{align*}
		P_{2^0} e^{\imu t |\nabla|}  g^M_{i,k,l}(r\theta) &= a_k r^{-\frac{d-2}{2}} b_{k,l}(\theta) \int_0^\infty \chi_{2^0}(\rho) e^{\imu t \rho} \hat{c}^{M,i}_{k,l}(\rho) J_{\frac{d+2k-2}{2}}(r \rho) \rho^{\frac{d}{2}} \dd \rho \\
		&= a_k b_{k,l}(\theta) T^{\frac{d+2k-2}{2}}_1(\hat{c}^{M,i}_{k,l})(t,r)
	\end{align*}
	with $T^{\frac{d+2k-2}{2}}_{1}$ from~\eqref{eq:DefTk1}. Property~\eqref{eq:BoundednessGoodFrame} of the good frame, Lemma~\ref{lem:EstimateTkp11}, and~\eqref{eq:L2SummabilityAngDec} thus yield
	\begin{align*}
		&\|P_{2^0} e^{\imu t |\nabla|}  g^M_{i,k,l} \|_{l^2_{i,k,l} L^q_t \cL^{p_0}_r L^\mu_\theta} \lesssim \|b_{k,l} T^{\frac{d+2k-2}{2}}_1(\hat{c}^{M,i}_{k,l})\|_{l^2_{i,k,l} L^q_t \cL^{p_0}_r L^\mu_\theta} \\
		&\lesssim \| T^{\frac{d+2k-2}{2}}_1(\hat{c}^{M,i}_{k,l})\|_{l^2_{i,k,l} L^q_t \cL^{p_0}_r} 
		\lesssim \| \hat{c}^{M,i}_{k,l} \|_{l^2_{i,k,l} L^2_\rho} \lesssim  \| \hat{c}^{M,i}_{k,l} \|_{l^2_{i,k,l} \cL^2_\rho} \lesssim \|g^{M}_i\|_{l^2_i L^2_x},
	\end{align*}
	i.e.~\eqref{eq:EstProbRegPrepForT} for every $M \in 2^\N$. Considering our previous reductions, this proves the assertion of the proposition.
\end{proof}

We now turn to the improved decay of solutions of the linear half-wave equation with randomized data originating in the physical-space component of the randomization. 

\begin{rem} 
\label{rem:DiscussionPhysicalSpaceRand}
The basic idea to apply the dispersive estimate is the same as in the Schr{\"o}dinger case, where the pure physical-space randomization was introduced. However, in contrast to the Schr{\"o}dinger equation, the dispersive estimate for the half-wave equation loses too much regularity in order to be useful here. In view of the frequency-space component of the randomization, we would like to employ the Klainerman-Tataru improved dispersive estimate instead. A direct application of this estimate yet fails as we lose the summability of all the different pieces of the decomposition underlying the randomization~\eqref{eq:DefRandomization}.

We thus rely on the variant of the Klainerman-Tataru improved dispersive estimate from Lemma~\ref{lem:ImprovedDispersiveEstimate}. As it only requires the data to be localized to an annulus, we can first sum the pieces of the frequency-space decomposition lying in such an annulus. The symbol of the localizer to that annulus is radial so that the remaining localization in frequency space commutes with the decomposition in the good frame. Consequently, we can sum up the pieces of the latter decomposition \emph{before} we apply Lemma~\ref{lem:ImprovedDispersiveEstimate}. This is the crucial point in order to extract the dispersive decay with an affordable loss of regularity and without losing the summability of the remaining pieces of the decomposition~\eqref{eq:DecPMf}.

The drawback of Lemma~\ref{lem:ImprovedDispersiveEstimate} is of course that it only applies for large times, where the meaning of large depends on the frequency. We thus split the time interval depending on the frequency and proceed for large times as explained above, while for small times we can apply Proposition~\ref{prop:ProbabilisticEstimateReg}.
\end{rem}

\begin{prop}
	\label{prop:ProbabilisticEstimateDecay}
	Let $d \geq 3$. Let $s \in \R$, $q,p \in [2,\infty)$, and $\sigma \geq 0$ such that $(q,p,\sigma)$ satisfies
	\begin{align*}
		\frac{d-1}{2} - \frac{1}{q} - \frac{d-1}{p} - \sigma > 0.
	\end{align*}
	Take $f \in H^s(\R^d)$ and let $f^\om$ be its randomization from~\eqref{eq:DefRandomization}. Fix $\lambda_0 \in \N$ with $\lambda_0 \geq \sqrt{d} + 4$ and $M_0 \in 2^{\N}$ with $M_0 \geq \lambda_0 + \sqrt{d}$. Then there exists a constant $C > 0$ such that
	\begin{align*}
		 \|\langle t \rangle^\sigma e^{\pm \imu t |\nabla|} P_{>M_0} f^\om \|_{L^\beta_\om L^q_t \dot{B}^s_{p,2}} \leq C \sqrt{\beta} \|f\|_{H^{s + \frac{1}{q} + \sigma}}
	\end{align*}
	for all $\beta \in [1,\infty)$.
\end{prop}

\begin{proof}
	Since $\Omega$ is a probability space, we only have to prove the assertion for $\beta \geq \max\{p,q\}\geq 2$. By time reversal symmetry, it is enough to prove the assertion for $e^{\imu t |\nabla|}$.
	
	In a first step, we apply Minkowski's inequality to infer
	\begin{align}
		&\|\langle t \rangle^\sigma e^{\imu t |\nabla|} P_{>M_0} f^\om \|_{L^\beta_\om L^q_t \dot{B}^s_{p,2}} 
		\leq \Big\|\Big(\sum_{M \in 2^\Z} M^{2s} \| \langle t \rangle^\sigma P_M e^{\imu t |\nabla|} P_{>M_0} f^\om \|_{L_t^q L_x^p}^2  \Big)^{\frac{1}{2}} \Big\|_{L^\beta_\om} \nonumber\\
		&\leq \Big\|\Big(\sum_{M \in 2^\Z} M^{2s} \| \langle t \rangle^\sigma P_M e^{\imu t |\nabla|} P_{>M_0} f^\om \|_{L_{[-M,M]}^q L_x^p}^2  \Big)^{\frac{1}{2}} \Big\|_{L^\beta_\om} \nonumber\\
		&\qquad + \Big\|\Big(\sum_{M \in 2^\Z} M^{2s} \| \langle t \rangle^\sigma P_M e^{\imu t |\nabla|} P_{>M_0} f^\om \|_{L_{I_M}^q L_x^p}^2  \Big)^{\frac{1}{2}} \Big\|_{L^\beta_\om}, \label{eq:EstProbDecTimeSplit}
	\end{align}
	where $I_M = (-\infty,-M) \cup (M, \infty)$. For the first term on the above right-hand side we get
	\begin{align}
		&\Big\|\Big(\sum_{M \in 2^\Z} M^{2s} \| \langle t \rangle^\sigma P_M e^{\imu t |\nabla|} P_{>M_0} f^\om \|_{L_{[-M,M]}^q L_x^p}^2  \Big)^{\frac{1}{2}} \Big\|_{L^\beta_\om} \nonumber\\
		&\lesssim \Big\|\Big(\sum_{M \in 2^\Z} M^{2(s+\sigma)} \| P_M e^{\imu t |\nabla|} P_{>M_0} f^\om \|_{L_{[-M,M]}^q L_x^p}^2  \Big)^{\frac{1}{2}} \Big\|_{L^\beta_\om} \lesssim \|e^{\imu t |\nabla|} P_{>M_0} f^\om\|_{L^\beta_\om \dot{B}^{s+\sigma}_{q,p,2}}. \label{eq:EstProbDecFiniteTime1}
	\end{align}
	Take $\delta_1 \in (0, \frac{1}{d(d-1)q})$ so small that $\frac{\delta_1}{2} \leq \frac{1}{2} - \frac{1}{(d-1)q} - \frac{1}{p}$ and define $p_0 \in [2,\infty)$ by
	\begin{align*}
		\frac{1}{p_0} = \frac{1 - \delta_1}{2} - \frac{1}{(d-1)q}.
	\end{align*}		
	This choice of $p_0$ implies that $p_0 \leq p$ as well as
	\begin{align*}
		\frac{1}{q} + \frac{d-1}{p_0} < \frac{d-1}{2} \qquad \text{and} \qquad -\frac{1}{q} - \frac{d}{p_0} + \frac{d}{2} = \frac{1}{(d-1)q} + \frac{d}{2} \delta_1 < \frac{1}{q},
	\end{align*}
	so that Proposition~\ref{prop:ProbabilisticEstimateReg} with $\mu = p$ yields
	\begin{align}
		\|e^{\imu t |\nabla|} P_{>M_0} f^\om\|_{L^\beta_\om \dot{B}^{s+\sigma}_{q,p,2}} \lesssim \sqrt{\beta} \|f\|_{H^{s + \sigma - \frac{1}{q} - \frac{d}{p_0} + \frac{d}{2}}} \lesssim \sqrt{\beta} \|f\|_{H^{s + \frac{1}{q} + \sigma}}. \label{eq:EstProbDecFiniteTime2}
	\end{align}
	
	It thus remains to estimate the second summand on the right-hand side of~\eqref{eq:EstProbDecTimeSplit}. 
	We only treat the interval $[M,\infty)$, the estimate on $(-\infty,-M]$ follows in the same way. We use the notation introduced in~\eqref{eq:DeffMom} to~\eqref{eq:DeffMikl}. As each $f^{M,\om}$ has frequency support in $\{M/2 < |\xi| < 2M\}$, we obtain
	\begin{align}
		&\Big\|\Big(\sum_{M \in 2^\Z} M^{2s} \| \langle t \rangle^\sigma P_M e^{\imu t |\nabla|} P_{>M_0} f^\om \|_{L_{[M, \infty)}^q L_x^p}^2  \Big)^{\frac{1}{2}} \Big\|_{L^\beta_\om} \nonumber \\
		&\lesssim \|M^s \langle t \rangle^\sigma  e^{\imu t |\nabla|} P_{>M_0} f^{M,\om}\|_{L^\beta_\om l^2_M L^q_{[M,\infty)} L^p_x} \nonumber\\
		&\lesssim \sqrt{\beta}  \|M^s \langle t \rangle^\sigma  e^{\imu t |\nabla|} P_{>M_0} P_j f^M_{i,k,l}\|_{l^2_M L^q_{[M,\infty)} L^p_x l^2_{i,j,k,l}}, \label{eq:EstProbDecLargeTime1}
	\end{align}
	where we also used Minkowski's inequality and the Khinthchine inequality~\cite[Lemma~3.1]{BT2008I}. We next note that
	\begin{align*}
		\sum_{j \in \Z^d} |e^{\imu t |\nabla|} P_{>M_0} P_j f^M_{i,k,l}|^2 = \sum_{\lambda = \lambda_0}^\infty \sum_{|j| \sim \lambda} |e^{\imu t |\nabla|} P_{>M_0} P_j f^M_{i,k,l}|^2,
	\end{align*}
	where summation over $|j| \sim \lambda$ means to sum over those $j \in \Z^d$ which satisfy $-\frac{1}{2} \leq |j| - \lambda < \frac{1}{2}$. Note that $P_{>M_0}P_j = 0$ for all $j \in \Z^d$ with $|j| \leq \lambda_0$ so that the summands for these $j$ vanish on the above left-hand side. We indicate the corresponding $l^2$-norms by the symbols $l^2_{\lambda \geq \lambda_0}$ and $l^2_{|j| \sim \lambda}$, respectively. Since the symbols of the multipliers $Q_{1,\lambda}$ form a partition of unity on $|\xi| \geq 1$ and $Q_{1,\lambda'} P_j = 0$ for every $j$ with $|j| \sim \lambda$ unless $\lambda \sim \lambda'$, i.e. $|\lambda - \lambda'| \leq \sqrt{d} + \frac{3}{2}$, we infer
	\begin{align}
		&\|M^s \langle t \rangle^\sigma  e^{\imu t |\nabla|} P_{>M_0} P_j f^M_{i,k,l}\|_{l^2_M L^q_{[M,\infty)} L^p_x l^2_{i,j,k,l}} \nonumber\\
		&= \Big\|M^s \langle t \rangle^\sigma  e^{\imu t |\nabla|} P_{>M_0} \sum_{\lambda' \sim \lambda} Q_{1,\lambda'} P_j f^M_{i,k,l} \Big\|_{l^2_M L^q_{[M,\infty)} L^p_x l^2_{i,k,l}l^2_{\lambda \geq \lambda_0} l^2_{|j|\sim\lambda}} \nonumber\\
		&\lesssim\|M^s \langle t \rangle^\sigma P_j e^{\imu t |\nabla|} P_{>M_0} Q_{1,\lambda} f^M_{i,k,l}\|_{l^2_M l^2_i l^2_{\lambda \geq 2} L^q_{[M,\infty)} L^p_x l^2_{k,l} l^2_{j}}, \label{eq:Splittingl2jNorm}
	\end{align}
	also employing Minkowski's inequality in the last step. In order to estimate the $l^2_{j}$-norm without losing summability in $l$ and $k$, we employ the pointwise bound
	\begin{equation*}
		\Big( \sum_{j \in \Z^d} |P_j g|^2 \Big)^{\frac{1}{2}} \lesssim (|\check{\psi}_0| \ast |g|^2)^{\frac{1}{2}}
\end{equation*}
from the proof of Lemma~2.8 in~\cite{KMV2019}, where $\check{\psi}_0$ denotes the inverse Fourier transform of $\psi_0$ from~\eqref{eq:DefPsij}.
	We obtain
	\begin{align}
		 &\|P_j e^{\imu t |\nabla|} P_{>M_0} Q_{1,\lambda} f^M_{i,k,l}\|_{L^p_x l^2_{k,l} l^2_{j}} \lesssim \| (|\check{\psi}_0| \ast |e^{\imu t |\nabla|} P_{>M_0} Q_{1,\lambda} f^M_{i,k,l}|^2)^{\frac{1}{2}}\|_{L^p_x l^2_{k,l}} \nonumber\\
		 &= \| |\check{\psi}_0| \ast \|e^{\imu t |\nabla|} P_{>M_0} Q_{1,\lambda} f^M_{i,k,l}\|_{l^2_{k,l}}^2\|_{L^{\frac{p}{2}}_x}^{\frac{1}{2}} \lesssim  \|\check{\psi}_0\|_{L^1_x}^{\frac{1}{2}} \| \|e^{\imu t |\nabla|} P_{>M_0} Q_{1,\lambda} f^M_{i,k,l}\|_{l^2_{k,l}}^2\|_{L^{\frac{p}{2}}_x}^{\frac{1}{2}} \nonumber\\
		 &\lesssim \|e^{\imu t |\nabla|} P_{>M_0} Q_{1,\lambda} f^M_{i,k,l}\|_{L^p_x l^2_{k,l}}. \label{eq:SquareFunctionOverAnnuli}
	\end{align}
	 Exploiting that the symbol of $e^{\imu t |\nabla|} P_{>M_0} Q_{1,\lambda}$, i.e.,
	\begin{align*}
		e^{\imu t |\xi|} \chi_{>M_0}(\xi) \zeta_{1,\lambda}(\xi) = e^{\imu t |\xi|} (1-\eta_0(|\xi|/M_0)) \zeta^0_{1}(|\xi|-\lambda)
	\end{align*}
	is radially symmetric, Theorem~3.10 in~\cite{SW71}, rescaling, and the definition of $f^M_{i,k,l}$  yield the representation
	\begin{align*}
		&e^{\imu t |\nabla|} P_{>M_0} Q_{1,\lambda} f^M_{i,k,l}(r \theta) = \Big[a_k M^{-\frac{d-2}{2}} r^{-\frac{d-2}{2}} \! \int_0^\infty \! e^{\imu M t \rho} (1-\eta_0(M \rho/M_0)) \zeta^0_{1}(M \rho - \lambda) \\
		&\hspace{20em} \cdot \hat{c}^{M,i}_{k,l}(\rho) J_{\frac{d+2k-2}{2}}(M r \rho) \rho^{\frac{d}{2}} \dd \rho \Big] b_{k,l}(\theta) \\
		&=: d^{M,i,\lambda}_{k,l}(t,r) b_{k,l}(\theta).
	\end{align*}
	Consequently, we deduce from~\eqref{eq:SumfMikl} that
	\begin{align*}
		&e^{\imu t |\nabla|} P_{>M_0} Q_{1,\lambda} P_M(\ph_i f)(r\theta) = \sum_{k =0}^\infty \sum_{l = 1}^{N_k} e^{\imu t |\nabla|} P_{>M_0} Q_{1,\lambda} f^M_{i,k,l}(r \theta) \\
		&= \sum_{k =0}^\infty \sum_{l = 1}^{N_k} d^{M,i,\lambda}_{k,l}(t,r) b_{k,l}(\theta).
	\end{align*}
	Hence, the functions $(d^{M,i,\lambda}_{k,l}(t,r))_{k,l}$ are the coefficients of the expansion of $e^{\imu t |\nabla|} P_{>M_0} Q_{1,\lambda} P_M(\ph_i f)(r \cdot)$ in the orthonormal basis $(b_{k,l})_{k,l}$ of $L^2(S^{d-1})$. Parseval's identity therefore yields
	\begin{align*}
		\|d^{M,i,\lambda}_{k,l}(t,r)\|_{l^2_{k,l}} = \|e^{\imu t |\nabla|} P_{>M_0} Q_{1,\lambda} P_M(\ph_i f)(r \cdot)\|_{L^2_\theta}.
	\end{align*}
	Using Minkowski's inequality, property~\eqref{eq:BoundednessGoodFrame} of the good frame, and H{\"o}lder's inequality on the sphere, we thus infer
	\begin{align}
		&\|e^{\imu t |\nabla|} P_{>M_0} Q_{1,\lambda} f^M_{i,k,l}\|_{L^p_x l^2_{k,l}} \lesssim \|d^{M,i,\lambda}_{k,l}  b_{k,l} \|_{\cL_r^p l^2_{k,l} L^p_\theta} \lesssim \|d^{M,i,\lambda}_{k,l} \|_{\cL_r^p l^2_{k,l}} \nonumber\\
		&\lesssim \|e^{\imu t |\nabla|} P_{>M_0} Q_{1,\lambda} P_M(\ph_i f)\|_{\cL^p_r L^2_\theta} \lesssim \|e^{\imu t |\nabla|} P_{>M_0} Q_{1,\lambda} P_M(\ph_i f)\|_{L^p_x}. \label{eq:EstProbDeckl}
	\end{align}
	We next note that $Q_{1,\lambda} P_M = 0$ if $\lambda + 1 < \frac{M}{2}$  or $\lambda - 1 > 2M$ as the symbols of $Q_{1,\lambda}$ and $P_M$ have disjoint supports in these cases. If $Q_{1,\lambda} P_M$ does not vanish for $M \in 2^\Z$ and $\lambda \in \N$ with $\lambda \geq 2$, we therefore have
	\begin{align*}
		3 \lambda \geq M \geq \frac{\lambda-1}{2} \geq \frac{\lambda}{4}.
	\end{align*}
	 Lemma~\ref{lem:ImprovedDispersiveEstimate} implies
	\begin{align*}
		\|e^{\imu t |\nabla|} P_{>M_0} Q_{1,\lambda} P_M(\ph_i f)\|_{L^p_x} \lesssim (\lambda^{-1} t)^{-\frac{d-1}{2}(1-\frac{2}{p})} \| Q_{1,\lambda} P_{>M_0} P_M(\ph_i f)\|_{L^{p'}_x}
	\end{align*}
	for all $t \geq \frac{\lambda}{4}$ and $\lambda \geq 2$. With the previous considerations on the relation between $M$ and $\lambda$ we conclude that 
	\begin{align*}
		\|e^{\imu t |\nabla|} P_{>M_0} Q_{1,\lambda} P_M(\ph_i f)\|_{L^p_x} \lesssim (M^{-1} t)^{-\frac{d-1}{2}(1-\frac{2}{p})} \| Q_{1,\lambda} P_{>M_0} P_M(\ph_i f)\|_{L^{p'}_x}
	\end{align*}		
	 for all $t \geq M$. 
	Combining this estimate with~\eqref{eq:EstProbDeckl},~\eqref{eq:SquareFunctionOverAnnuli}, and~\eqref{eq:Splittingl2jNorm}, we obtain
	\begin{align}
		&\|M^s \langle t \rangle^\sigma  e^{\imu t |\nabla|} P_{>M_0} P_j f^M_{i,k,l}\|_{l^2_M L^q_{[M,\infty)} L^p_x l^2_{i,j,k,l}} \nonumber\\
		&\lesssim \|M^s \langle t \rangle^\sigma e^{\imu t |\nabla|} P_{>M_0} Q_{1,\lambda} P_M(\ph_i f)\|_{l^2_M l^2_i l^2_{\lambda \geq 2} L^q_{[M,\infty)} L^p_x} \nonumber\\
		&\lesssim \|M^s t^\sigma (M^{-1} t)^{-\frac{d-1}{2}(1-\frac{2}{p})} \| Q_{1,\lambda} P_{>M_0} P_M(\ph_i f)\|_{L^{p'}_x}\|_{l^2_M l^2_i l^2_{\lambda \geq 2} L^q_{[M,\infty)}} \nonumber \\
		&\lesssim \|M^{s+\sigma} \| (M^{-1} t)^{-\frac{d-1}{2}+\frac{d-1}{p}+\sigma}\|_{L^q_{[M,\infty)}} \| Q_{1,\lambda} P_{>M_0} P_M(\ph_i f)\|_{L^{p'}_x}\|_{l^2_M l^2_i l^2_{\lambda \geq 2}} \nonumber \\
		&\lesssim \|M^{s+\frac{1}{q}+\sigma}  Q_{1,\lambda} P_{>M_0} P_M(\ph_i f)\|_{l^2_M l^2_i l^2_{\lambda \geq 2}L^{p'}_x}, \label{eq:EstProbDecIntTime}
	\end{align}
	where we exploited the assumption on $q$, $p$, and $\sigma$ in the last step. Inserting~\eqref{eq:EstProbDecIntTime} into~\eqref{eq:EstProbDecLargeTime1} and applying Corollary~\ref{cor:HsEstimatePhysicalSpace}~\ref{it:HsMismatchLPAndAn}, we arrive at
	\begin{align}
		\Big\|\Big(\sum_{M \in 2^\Z} M^{2s} \| \langle t \rangle^\sigma P_M e^{\imu t |\nabla|} P_{>M_0} f^\om \|_{L_{[M, \infty)}^q L_x^p}^2  \Big)^{\frac{1}{2}} \Big\|_{L^\beta_\om} \lesssim \sqrt{\beta} \|f\|_{H^{s+\frac{1}{q}+\sigma}}. \label{eq:EstProbDecLargeTimeFinal}
	\end{align}
	The estimate on $(-\infty,-M]$ is derived in the same way so that~\eqref{eq:EstProbDecLargeTimeFinal}, \eqref{eq:EstProbDecFiniteTime1} and \eqref{eq:EstProbDecFiniteTime2}, and~\eqref{eq:EstProbDecTimeSplit} yield the assertion of the proposition.
\end{proof}

Interpolating between the different advantages of the randomization, i.e. between the estimates from Proposition~\ref{prop:ProbabilisticEstimateReg} and Proposition~\ref{prop:ProbabilisticEstimateDecay}, we obtain the desired $L^1_t L^\infty_x$-estimate for the linear half-wave flow with randomized data.
\begin{prop}
	\label{prop:ProbabilisticL1LinftyEstiamte}
	Let $d \geq 4$ and $s, s' \in \R$ with $s - s' > \frac{d+1}{2(d-1)}$. Take $f \in H^s(\R^d)$ and let $f^\om$ denote its randomization from~\eqref{eq:DefRandomization}. Fix a dyadic integer $M_0 \in 2^\N$ with $M_0 \geq \lceil \sqrt{d} + 4 \rceil + \sqrt{d}$. Then there exists a constant $C > 0$ such that
	\begin{align*}
		\| |\nabla|^{s'} e^{\pm \imu t |\nabla|} P_{>M_0} f^\om \|_{L^\beta_\om L^1_t L^\infty_x} \leq C \sqrt{\beta} \|f\|_{H^s}
	\end{align*}
	for all $\beta \in [1,\infty)$.
\end{prop}

\begin{proof}
	We only prove the assertion for $e^{\imu t |\nabla|}$, the proof for $e^{-\imu t |\nabla|}$ works in the same way.
	Take $0 < \delta \ll 1$ such that
	\begin{align}
	\label{eq:Conditiondelta}
		\frac{d+1}{2(d-1)} + s' + \delta + \frac{d^2}{2(d-1)} \delta \leq s
	\end{align}
	and set $\eta = 2 \delta$. We first apply Sobolev's embedding to infer
	\begin{align}
	\label{eq:EstHalfWaveL1LInftyFirstReduction}
		&\| |\nabla|^{s'} e^{\imu t |\nabla|} P_{>M_0} f^\om \|_{L^\beta_\om L^1_t L^\infty_x} \lesssim \| \langle \nabla \rangle^{\delta} |\nabla|^{s'} e^{\imu t |\nabla|} P_{>M_0} f^\om \|_{L^\beta_\om L^1_t L^{\frac{2d}{\delta}}_x} \nonumber \\
		&\lesssim \|  e^{\imu t |\nabla|} P_{>M_0} f^\om \|_{L^\beta_\om L^1_t \dot{B}^{s'}_{\frac{2d}{\delta},2}} + \| e^{\imu t |\nabla|} P_{>M_0} f^\om \|_{L^\beta_\om L^1_t \dot{B}^{s' + \delta}_{\frac{2d}{\delta},2}}.
	\end{align}
	Let $\tilde{s} \in \R$ with $\tilde{s} \leq s' + \delta$. We next note that
	\begin{align*}
		&\|g\|_{\dot{B}^{\tilde{s}}_{\frac{2d}{\delta},2}} = \Big(\sum_{N \in 2^\Z} N^{2\tilde{s}} \|P_N g\|_{L^{\frac{2d}{\delta}}_x}^2 \Big)^{\frac{1}{2}} \\
		&= \Big(\sum_{N \in 2^\Z} (N^{\frac{d}{d-1}(1+\eta) + 2 \tilde{s}} \|P_N g\|_{L^{\frac{2d}{\delta}}_x}^2)^{\frac{d-3-\eta}{d-2}} (N^{-\frac{d(d-3-\eta)}{d-1} + 2 \tilde{s}} \|P_N g\|_{L^{\frac{2d}{\delta}}_x}^2)^{\frac{1+\eta}{d-2}} \Big)^{\frac{1}{2}} \\
		&\leq \Big(\Big(\sum_{N \in 2^\Z} N^{\frac{d}{d-1}(1+\eta) + 2 \tilde{s}} \|P_N g\|_{L^{\frac{2d}{\delta}}_x}^2\Big)^{\frac{d-3-\eta}{d-2}} \Big(\sum_{N \in 2^\Z} N^{-\frac{d(d-3-\eta)}{d-1} + 2 \tilde{s}} \|P_N g\|_{L^{\frac{2d}{\delta}}_x}^2\Big)^{\frac{1+\eta}{d-2}} \Big)^{\frac{1}{2}} \\
		&= \|g\|_{\dot{B}^{\frac{d}{2(d-1)}(1+\eta)+\tilde{s}}_{\frac{2d}{\delta},2}}^{\frac{d-3-\eta}{d-2}} \|g\|_{\dot{B}^{-\frac{d(d-3-\eta)}{2(d-1)} + \tilde{s}}_{\frac{2d}{\delta},2}}^{\frac{1+\eta}{d-2}}.
	\end{align*}
	Combining this inequality with H{\"o}lder's inequality in time, we derive
	\begin{align}
		&\| e^{\imu t |\nabla|} P_{>M_0} f^\om \|_{L^\beta_\om L^1_t \dot{B}^{\tilde{s}}_{\frac{2d}{\delta},2}} \label{eq:EstL1LInftySplit} \\ 
		&\lesssim \Big\| \langle t \rangle^{\frac{1+\delta}{2}}  \|e^{\imu t |\nabla|} P_{>M_0} f^\om\|_{\dot{B}^{\frac{d}{2(d-1)}(1+\eta)+\tilde{s}}_{\frac{2d}{\delta},2}}^{\frac{d-3-\eta}{d-2}} \|e^{\imu t |\nabla|} P_{>M_0} f^\om\|_{\dot{B}^{-\frac{d(d-3-\eta)}{2(d-1)} + \tilde{s}}_{\frac{2d}{\delta},2}}^{\frac{1+\eta}{d-2}} \Big\|_{L^\beta_\om L^2_t} \nonumber\\
		&\lesssim  \|e^{\imu t |\nabla|} P_{>M_0} f^\om\|_{L^\beta_\om L^2_t \dot{B}^{\frac{d}{2(d-1)}(1+\eta)+\tilde{s}}_{\frac{2d}{\delta},2}}^{\frac{d-3-\eta}{d-2}}  \| \langle t \rangle^{\frac{1+\delta}{1+\eta}\frac{d-2}{2}}   e^{\imu t |\nabla|} P_{>M_0} f^\om\|_{L^\beta_\om L^2_t \dot{B}^{-\frac{d(d-3-\eta)}{2(d-1)} + \tilde{s}}_{\frac{2d}{\delta},2}}^{\frac{1+\eta}{d-2}}. \nonumber 
	\end{align}
	Applying Proposition~\ref{prop:ProbabilisticEstimateReg} with $p_0 = \frac{2(d-1)}{(d-2)(1-\delta)}$ and $\mu = p = \frac{2d}{\delta}$, we obtain
	\begin{align}
	\label{eq:EstL1LInftyReg}
		\|e^{\imu t |\nabla|} P_{>M_0} f^\om\|_{L^\beta_\om L^2_t \dot{B}^{\frac{d}{2(d-1)}(1+\eta)+\tilde{s}}_{\frac{2d}{\delta},2}} \lesssim \sqrt{\beta} \|f\|_{H^{\frac{d}{2(d-1)}(1+\eta)+\tilde{s} - \frac{1}{2} - \frac{d}{p_0} + \frac{d}{2}}} \lesssim \sqrt{\beta} \|f\|_{H^s},
	\end{align}
	where we also used that
	\begin{align*}
		&\frac{d}{2(d-1)}(1+\eta)+\tilde{s} - \frac{1}{2} - \frac{d}{p_0} + \frac{d}{2} = \frac{d(1+\eta)}{2(d-1)}+\tilde{s} - \frac{1}{2} - \frac{d(d-2)(1-\delta)}{2(d-1)} + \frac{d}{2} \\
		&= \frac{d-(d-1)-d(d-2)+d(d-1)}{2(d-1)} + \tilde{s} + \frac{d}{2(d-1)}(\eta + (d-2)\delta)\\
		& = \frac{d+1}{2(d-1)} + \tilde{s}  + \frac{d^2}{2(d-1)} \delta \leq \frac{d+1}{2(d-1)} + s' + \delta  + \frac{d^2}{2(d-1)} \delta \leq s
	\end{align*}
	by~\eqref{eq:Conditiondelta}, as $\eta = 2 \delta$ and $\tilde{s} \leq s' + \delta$. On the other hand, since
	\begin{align*}
		&\frac{d-1}{2} - \frac{1}{2} - \frac{d-1}{\frac{2d}{\delta}} - \frac{1+\delta}{1+\eta}\frac{d-2}{2} = \frac{d-2}{2}(1-\frac{1+\delta}{1+\eta}) - \frac{d-1}{2d}\delta \\
		&= \frac{d-2}{2} \frac{\delta}{1+\eta} - \frac{d-1}{2d}\delta \geq \frac{d-2}{4}\delta - \frac{d-1}{8} \delta = \frac{d-3}{8} \delta > 0,
	\end{align*}
	Proposition~\ref{prop:ProbabilisticEstimateDecay} with $\lambda_0 = \lceil \sqrt{d} + 4 \rceil$ yields
	\begin{align}
	\label{eq:EstL1LInftyDec}
		\| \langle t \rangle^{\frac{1+\delta}{1+\eta} \frac{d-2}{2}}   e^{\imu t |\nabla|} P_{>M_0} f^\om\|_{L^\beta_\om L^2_t \dot{B}^{-\frac{d(d-3-\eta)}{2(d-1)} + \tilde{s}}_{\frac{2d}{\delta},2}} 
		&\lesssim \sqrt{\beta} \|f\|_{H^{-\frac{d(d-3-\eta)}{2(d-1)} + \tilde{s} + \frac{1}{2} + \frac{1+\delta}{1+\eta}\frac{d-2}{2}}} \nonumber\\
		&\lesssim \sqrt{\beta} \|f\|_{H^s},
	\end{align}
	where we exploited
	\begin{align*}
		&-\frac{d(d-3-\eta)}{2(d-1)} + \tilde{s} + \frac{1}{2} + \frac{1+\delta}{1+\eta}\frac{d-2}{2} \\
		&=\frac{-d(d-3) + (d-1) + (d-2)(d-1)}{2(d-1)} + \tilde{s} + \frac{d}{2(d-1)} \eta + \frac{\delta - \eta}{1+\eta} \frac{d-2}{2} \\
		&= \frac{d+1}{2(d-1)} + \tilde{s} - \frac{\delta}{1+\eta} \frac{d-2}{2} + \frac{d}{d-1} \delta
		<\frac{d+1}{2(d-1)} + s' + \delta  + \frac{d^2}{2(d-1)} \delta \leq s
	\end{align*}
	by~\eqref{eq:Conditiondelta} in the last step. Inserting~\eqref{eq:EstL1LInftyReg} and~\eqref{eq:EstL1LInftyDec} into~\eqref{eq:EstL1LInftySplit}, we finally arrive at
	\begin{align*}
		\| e^{\imu t |\nabla|} P_{>M_0} f^\om \|_{L^\beta_\om L^1_t \dot{B}^{\tilde{s}}_{\frac{2d}{\delta},2}} \lesssim \sqrt{\beta} \|f\|_{H^s}.
	\end{align*}
	Applying this estimate with $\tilde{s} = s'$ and $\tilde{s} = s' + \delta$, respectively, we obtain the assertion from~\eqref{eq:EstHalfWaveL1LInftyFirstReduction}.
\end{proof}

The $L^1_t L^\infty_x$-estimate for solutions of the half-wave equation immediately transfers to solutions of the wave equation with randomized data. The resulting almost sure $L^1_t L^\infty_x$-estimate is the crucial ingredient in the proof of the almost sure scattering result in Theorem~\ref{thm:AlmostSureScattering}.
\begin{cor}
	\label{cor:L1LinftyEstimateWavePropagator}
	Let $d \geq 4$ and $s > \frac{d+1}{2(d-1)}$. Take $(f_0, f_1) \in H^s(\R^d) \times H^{s-1}(\R^d)$ and let $(f_0^\om, f_1^\om)$ denote the randomizations of $f_0$ and $f_1$ from~\eqref{eq:DefRandomization}, respectively. Fix $M_0 \in 2^\N$ with $M_0 \geq \lceil \sqrt{d} + 4 \rceil + \sqrt{d}$. Then there is a constant $C > 0$ such that
	\begin{align*}
		\|S(t)(P_{>M_0} f_0^\om, P_{>M_0}f_1^\om)\|_{L^\beta_\om L^1_t L^\infty_x} \leq C \sqrt{\beta} \|(f_0, f_1)\|_{H^s \times H^{s-1}}
	\end{align*}
	for all $\beta \in [1,\infty)$.
\end{cor}

\begin{proof}
	The identity
	\begin{align*}
		S(t)(P_{>M_0} f_0^\om, P_{>M_0} f_1^\om) &= \cos(t |\nabla|) P_{>M_0} f_0^\om + \frac{\sin(t |\nabla|)}{|\nabla|} P_{>M_0} f_1^\om \\
		&= \frac{e^{\imu t |\nabla|} + e^{-\imu t |\nabla|}}{2} P_{>M_0} f_0^\om + \frac{e^{\imu t |\nabla|} - e^{-\imu t |\nabla|}}{2 \imu |\nabla|} P_{>M_0} f_1^\om
	\end{align*}
	implies that
	\begin{align*}
		&\|S(t)(P_{>M_0} f_0^\om, P_{>M_0} f_1^\om)\|_{L^\beta_\om L^1_t L^\infty_x}   \\
		&\lesssim \|e^{\imu t |\nabla|} P_{>M_0} f_0^\om\|_{L^\beta_\om L^1_t L^\infty_x} + \|e^{-\imu t |\nabla|} P_{>M_0} f_0^\om\|_{L^\beta_\om L^1_t L^\infty_x} \\
		&\qquad + \||\nabla|^{-1} e^{\imu t |\nabla|} P_{>M_0} f_1^\om\|_{L^\beta_\om L^1_t L^\infty_x} + \||\nabla|^{-1} e^{-\imu t |\nabla|} P_{>M_0} f_1^\om\|_{L^\beta_\om L^1_t L^\infty_x}.
	\end{align*}
	Applying Proposition~\ref{prop:ProbabilisticL1LinftyEstiamte} with regularity parameters $(s, 0)$ to the first two summands on the right-hand side and with regularity parameters $(s-1,-1)$ to the last two summands, the assertion follows.
\end{proof}

\section{Almost sure scattering}
\label{sec:AlmostSureScattering}

In this section we prove the main theorem of this article, i.e. the almost sure scattering result for the defocusing cubic energy-critical nonlinear wave equation with supercritical data. In particular, we only consider dimension $d = 4$ from now on. We follow the general scheme introduced in~\cite{DLM2020}.

The first step consists in the derivation of a local wellposedness theory for the forced equation~\eqref{eq:CubicNLWForced}, which can be done via Strichartz estimates in the same way as for the unforced equation. We state the local wellposedness result in the form of Lemma~3.1 from~\cite{DLM2020}.

\begin{lem}
	\label{lem:NLWLocWP}
	Let $(v_0,v_1) \in \dot{H}^1(\R^4) \times L^2(\R^4)$ and $F \in L^3(\R, L^6(\R^4))$. Then there is a unique solution 
	\begin{align*}
		(v, \partial_t v) \in (C(I_{*}, \dot{H}^1(\R^4)) \cap L^3(I_{*}, L^6(\R^4))) \times C(I_{*}, L^2(\R^4))
	\end{align*}
	of~\eqref{eq:CubicNLWForced} on the maximal interval of existence $I_{*} = (T_{-}, T_+)$.
	
	We have the blow-up criterion that $T_+ < \infty$ implies $\|v\|_{L^3_{[0,T_+)}L^6_x} = \infty$ and analogously for $T_{-}$. Moreover, if $v$ is a global solution of~\eqref{eq:CubicNLWForced} satisfying $\|v\|_{L^3_{\R}L^6_x} < \infty$, then $v$ scatters to free waves as $t \rightarrow \pm \infty$.
\end{lem}

Combined with the global wellposedness and scattering theory of~\eqref{eq:CubicNLW} at energy regularity, one then derives the following conditional scattering result, see Theorem~1.3 in~\cite{DLM2020}.

\begin{prop}
	\label{prop:CondScattering}
	Let $(v_0, v_1) \in \dot{H}^1(\R^4) \times L^2(\R^4)$ and $F \in L^3(\R, L^6(\R^4))$. Let $v$ be the unique solution of~\eqref{eq:CubicNLWForced} on the maximal interval of existence $(T_{-}, T_+)$. Assume that
	\begin{align*}
		\sup_{t \in (T_{-}, T_+)} E(v(t)) < \infty.
	\end{align*}
	Then $(T_{-}, T_+) = \R$, i.e. $v$ exists for all times, and $v$ satisfies $\|v\|_{L^3_t L^6_x} < \infty$. Consequently, the solution $v$ scatters both forward and backward in time, i.e. there are $(v_0^\pm, v_1^\pm) \in \dot{H}^1(\R^4) \times L^2(\R^4)$ such that
	\begin{align*}
		\lim_{t \rightarrow \pm \infty} \|(v(t), \partial_t v(t)) - (S(t)(v_0^\pm, v_1^\pm), \partial_t S(t)(v_0^\pm, v_1^\pm)) \|_{\dot{H}^1 \times L^2} = 0.
	\end{align*}
\end{prop}

This conditional scattering result reduces the question of scattering to an a priori energy bound for solutions of~\eqref{eq:CubicNLWForced}. If the forcing belongs to $L^3_t L^6_x \cap L^1_t L^\infty_x$, we can bound the energy increment in terms of the energy, implying such an a priori bound.

\begin{thm}
	\label{thm:ScatteringForcedEquation}
	Let $(v_0, v_1) \in \dot{H}^1(\R^4) \times L^2(\R^4)$ and
	\begin{align*}
		F \in L^3(\R, L^6(\R^4)) \cap L^1(\R, L^\infty(\R^4)).
	\end{align*}
	Let $v$ be the unique solution of~\eqref{eq:CubicNLWForced} on its maximal interval of existence $(T_{-}, T_+)$. Then
	\begin{align*}
		\sup_{t \in (T_{-}, T_+)} E(v(t)) < \infty.
	\end{align*}
	In particular, $(T_{-}, T_+) = \infty$ and $v$ scatters both forward and backward in time.
\end{thm}

\begin{proof}
	Differentiating the energy of $v$ and integrating by parts, we infer
	\begin{equation}
		\label{eq:TimeDerEnergy}
		\partial_t E(v(t)) = \int_{\R^4} (v^3(t) - (v + F)^3(t)) \partial_t v(t) \dd x.
	\end{equation}
		Set
		\begin{align*}
			A_{T_0}(T) = 1 + \sup_{t \in (T_0,T)}E(v(t))
		\end{align*}
		for all $T_0, T \in (T_{-}, T_+)$ with $T > T_0$. Fixing such $T_0$ and $T$ for a moment, we get
		\begin{equation}
		\label{eq:EstAT0T1}
			A_{T_0}(T) \leq 1 + E(v(T_0)) + \int_{T_0}^T |\partial_t E(v(t))| \dd t.
		\end{equation}
		Since
		\begin{align*}
			|v^3 - (v + F)^3| \lesssim |v|^2 |F| + |F|^3,
\end{align*}				
		we obtain from~\eqref{eq:TimeDerEnergy}
		\begin{align*}
			\|\partial_t E(v(t))\|_{L^1_t L^1_x} &\lesssim \| v^3 - (v + F)^3\|_{L^1_t L^2_x} \|\partial_t v\|_{L^\infty_t L^2_x} \\
			&\lesssim (\| v^2 F \|_{L^1_t L^2_x} + \|F^3\|_{L^1_t L^2_x}) \sup_{t \in (T_0,T)} E(v(t))^{\frac{1}{2}} \\
			 &\lesssim (\|v\|_{L^\infty_t L^4_x}^2 \|F\|_{L^1_t L^\infty_x} + \|F\|_{L^3_t L^6_x}^3) \sup_{t \in (T_0,T)} E(v(t))^{\frac{1}{2}} \\
			 &\lesssim \|F\|_{L^1_t L^\infty_x} \sup_{t \in (T_0,T)}E(v(t)) +  \|F\|_{L^3_t L^6_x}^3 \sup_{t \in (T_0,T)} E(v(t))^{\frac{1}{2}} \\
			 &\lesssim (\|F\|_{L^1_t L^\infty_x} + \|F\|_{L^3_t L^6_x}^3) A_{T_0}(T),
		\end{align*}
		where all the space-time norms are taken over $(T_0,T) \times \R^4$. Inserting this estimate into~\eqref{eq:EstAT0T1}, we get
		\begin{equation}
			\label{eq:EstAT0T2}
			A_{T_0}(T) \leq 1 + E(v(T_0)) + C (\|F\|_{L^1_t L^\infty_x} + \|F\|_{L^3_t L^6_x}^3) A_{T_0}(T)
		\end{equation}
		with a constant $C > 0$ independent of $T_0$ and $T$. Fix $\delta > 0$ such that $C \delta < \frac{1}{2}$. The assumptions and the dominated convergence theorem yield times $0 = t_0 < t_1 < \ldots < t_n < t_{n+1} = T_+$ such that
		\begin{align*}
			\|F\|_{L^1_{I_j} L^\infty_x} + \|F\|_{L^3_{I_j} L^6_x}^3 \leq \delta
		\end{align*}
		for all $j \in \{0, 1, \ldots, n\}$, where $I_j = (t_{j}, t_{j+1})$. Applying estimate~\eqref{eq:EstAT0T2} on the intervals $I_j$, we thus deduce
		\begin{align*}
			A_{t_{j}}(t_{j+1}) \leq 1 + E(v(t_j)) + C \delta A_{t_{j}}(t_{j+1}), \qquad
			A_{t_{n}}(T) \leq 1 + E(v(t_n)) + C \delta A_{t_{n}}(T)
		\end{align*}
		and thus
		\begin{align*}
			A_{t_{j}}(t_{j+1}) \leq 2 + 2 E(v(t_j)), \qquad A_{t_{n}}(T) \leq 2 + 2 E(v(t_n))
		\end{align*}
		for all $j \in \{0, 1, \ldots, n-1\}$ and $T \in (t_n, T_+)$ as $C \delta < \frac{1}{2}$. Letting $T \rightarrow T_+$, we also get
		\begin{align*}
			A_{t_{n}}(t_{n+1}) = A_{t_n}(T_+) \leq 2 + 2 E(v(t_n)).
		\end{align*}
		Using $E(v(t_j)) \leq A_{t_{j-1}}(t_j)$ for all $j \in \{1, \ldots, n+1\}$, we derive inductively that
		\begin{align*}
			A_{t_{j}}(t_{j+1}) \leq 2^{j+2} - 2 + 2^{j+1} E(v(0))
		\end{align*}
		for all $j \in \{0, 1, \ldots, n\}$. Consequently,
		\begin{align*}
			\sup_{t \in (0, T_+)} E(v(t)) &< \max\{A_{t_{j}}(t_{j+1}) \colon j \in \{0,1, \ldots, n\}\} \\
				&= 2^{n+2} - 2 + 2^{n+1} E(v(0)) < \infty.
		\end{align*}
		Arguing analogously on $(T_{-}, 0)$, we arrive at
		\begin{align*}
			\sup_{t \in (T_{-}, T_+)} E(v(t)) < \infty.
		\end{align*}
		Proposition~\ref{prop:CondScattering} now implies that $v$ exists globally and scatters both forward and backward in time.
\end{proof}

We point out that in the application to the almost sure scattering result for~\eqref{eq:CubicNLWRand}, the forcing is the high frequency part of the solution of the linear wave equation with rough and randomized data. Theorem~\ref{thm:ScatteringForcedEquation} with its strong assumption of $F \in L^1_t L^\infty_x$ is thus only sufficient for our purposes because we have a corresponding almost sure $L^1_t L^\infty_x$-bound for the linear flow of the wave equation with randomized data.

The combination of Theorem~\ref{thm:ScatteringForcedEquation} and Corollary~\ref{cor:L1LinftyEstimateWavePropagator} now immediately implies Theorem~\ref{thm:AlmostSureScattering}.

\emph{Proof of Theorem~\ref{thm:AlmostSureScattering}:} We first recall that $u$ is a solution of the cubic NLW~\eqref{eq:CubicNLW} with initial data $(f_0^\om, f_1^\om)$ if and only if 
\begin{align*}
	v(t) = u(t) - S(t)(P_{>8} f_0^\om, P_{>8} f_1^\om)
\end{align*}
is a solution of the forced equation~\eqref{eq:CubicNLWForced} with initial data $(P_{\leq 8} f_0^\om, P_{\leq 8} f_1^\om)$ and forcing $F(t) = S(t)(P_{>8} f_0^\om, P_{>8} f_1^\om)$.

Fix $0 < \delta \ll 1$. Applying Proposition~\ref{prop:ProbabilisticEstimateReg} with $q = 3$, $p_0 = \frac{18}{7-\delta}$, $p = \mu = 6$, and $s = \frac{1}{9} + \frac{2}{9} \delta$ respectively $s = -\frac{8}{9} + \frac{2}{9} \delta$, we obtain
\begin{align*}
	\|S(t)(P_{>8}f_0^\om, P_{>8} f_1^\om)\|_{L_\om^\beta L^3_t L^6_x} \lesssim \sqrt{\beta} \|(f_0, f_1)\|_{H^{\frac{1}{9} + \frac{2}{9}\delta} \times H^{-\frac{8}{9} + \frac{2}{9}\delta}} \lesssim \sqrt{\beta} \|f\|_{H^s \times H^{s-1}}.
\end{align*}
We now set $\lambda_0 = 6$ and $M_0 = 8$. Corollary~\ref{cor:L1LinftyEstimateWavePropagator} then implies that
\begin{align*}
	\|S(t)(P_{>8}f_0^\om, P_{>8} f_1^\om)\|_{L^\beta_\om L^1_t L^\infty_x} \lesssim \sqrt{\beta} \|f\|_{H^s \times H^{s-1}}
\end{align*}
for all $\beta \geq 1$ as $\frac{d+1}{2(d-1)} = \frac{5}{6}$. Consequently,
\begin{align*}
	S(t)(P_{>8}f_0^\om, P_{>8} f_1^\om) \in L^3(\R, L^6(\R^4)) \cap L^1(\R, L^\infty(\R^4))
\end{align*}
for almost all $\om \in \Omega$. Since $(P_{\leq 8} f_0^\om, P_{\leq 8} f_1^\om)$ belongs to $\dot{H}^1(\R^4) \times L^2(\R^4)$ for almost all $\om \in \Omega$, the assertion of Theorem~\ref{thm:AlmostSureScattering} now follows from Theorem~\ref{thm:ScatteringForcedEquation} and Proposition~\ref{prop:CondScattering}. \hfill \qed

\vspace{1em}

\textbf{Acknowledgment.} I want to thank Sebastian Herr for valuable discussions and helpful feedback on this project. Funded by the Deutsche Forschungsgemeinschaft (DFG, German Research Foundation) – SFB 1283/2 2021 – 317210226.

\bibliographystyle{abbrv}
\bibliography{SupercriticalCubicNLW} 
 
\end{document}